\documentclass[letterpaper,10pt,reqno,onefignum,onetabnum]{amsart}
\usepackage[english]{babel}
\usepackage{amsmath}
\usepackage{amsthm}
\usepackage{verbatim}
\usepackage{mathrsfs}
\usepackage{mathabx}
\usepackage{bm} 
\usepackage{fancybox, calc}
\usepackage[dvipsnames]{xcolor}
\usepackage{hyperref}
\hypersetup{
     colorlinks = true,
     linkcolor = OliveGreen,
     anchorcolor = OliveGreen,
     citecolor = OliveGreen,
     filecolor = OliveGreen,
     urlcolor = OliveGreen
     }
\usepackage{algorithm}
\usepackage{algorithmic}
\usepackage{float}
\usepackage{lipsum}
\usepackage{amsfonts}
\usepackage{amssymb}
\usepackage{graphicx}
\usepackage{epstopdf}
\usepackage{subcaption}
\usepackage{cases}
\usepackage{multirow}
\usepackage[utf8]{inputenc}
\usepackage{csquotes}
\ifpdf
  \DeclareGraphicsExtensions{.eps,.pdf,.png,.jpg}
\else
  \DeclareGraphicsExtensions{.eps}
\fi
\usepackage{verbatim}
\usepackage{mathrsfs}
\usepackage{bm}

\newtheorem{theorem}{Theorem}[section]
\newtheorem{proposition}[theorem]{Proposition}

\newtheorem{problem}[theorem]{Problem}
\newtheorem{algorith}[theorem]{Algorithm}

\newtheorem{remark}[theorem]{Remark}

\usepackage{lipsum}
\usepackage{amsfonts}
\usepackage{amssymb}
\usepackage{graphicx}
\usepackage{epstopdf}

\usepackage{cases}
\usepackage{multirow}
\usepackage{algorithmic}
\usepackage{tikz}
\usetikzlibrary{matrix}
\usetikzlibrary{arrows}
\ifpdf
  \DeclareGraphicsExtensions{.eps,.pdf,.png,.jpg}
\else
  \DeclareGraphicsExtensions{.eps}
\fi
\usepackage{verbatim}
\usepackage{mathrsfs}
\usepackage{bm} 
\usepackage{color}

\newcommand{\NN}{\mathbb N}

\newcommand{\M}{\mathcal M}
\newcommand{\Y}{\mathcal Y}

\newcommand{\GG}{\mathcal G}

\newcommand{\HH}{{{\mathcal{H}}}}

\newcommand{\TT}{{\mathbb T}}
\newcommand{\gr}{\ensuremath{\operatorname{gra}}}

\newcommand{\Id}{{\rm Id}}

\newcommand{\emp}{\ensuremath{{\varnothing}}}

\newcommand{\ran}{\textnormal{ran}\,}

\newcommand{\Fix}{\ensuremath{\operatorname{Fix}}}

\newcommand{\mini}{\mathop{\mathrm{minimize}}}

\newcommand{\scal}[2]{{\left\langle{{#1}\mid{#2}}\right\rangle}}

\newcommand{\menge}[2]{\big\{{#1}~\big |~{#2}\big\}} 
\newcommand{\pinf}{\ensuremath{{+\infty}}}

\newcommand{\midd}{\ensuremath{\operatorname{mid}}}
\newcommand{\RR}{\ensuremath{\mathbb{R}}}
\newcommand{\RP}{\ensuremath{\left[0,+\infty\right[}}
\newcommand{\RM}{\ensuremath{\left]-\infty,0\right]}}

\newcommand{\RPP}{\ensuremath{\left]0,+\infty\right[}}

\newcommand{\infconv}{\ensuremath{\mbox{\small$\,\square\,$}}}

\newcommand{\RX}{\ensuremath{\left]-\infty,+\infty\right]}}

\newcommand{\prox}{\ensuremath{\text{\rm prox}}}

\newcommand{\argminE}{\mathop{\mathrm{argmin}}}
\usepackage{geometry}
\usepackage[
    backend=biber,
    style=numeric,
  ]{biblatex}
\numberwithin{equation}{section}

\DeclareFontEncoding{FMS}{}{}
\DeclareFontSubstitution{FMS}{futm}{m}{n}
\DeclareFontEncoding{FMX}{}{}
\DeclareFontSubstitution{FMX}{futm}{m}{n}
\DeclareSymbolFont{fouriersymbols}{FMS}{futm}{m}{n}
\DeclareSymbolFont{fourierlargesymbols}{FMX}{futm}{m}{n}
\DeclareMathDelimiter{\nr}{\mathord}{fouriersymbols}{152}{fourierlargesymbols}{147}

\DeclareMathDelimiter{\nr}{\mathord}{fouriersymbols}{152}{fourierlargesymbols}{147}

\title[{Primal-Dual Partial Inverse 
Algorithm}]{A Primal-Dual Partial Inverse 
Algorithm for Constrained 
Monotone Inclusions: Applications to Stochastic Programming
and Mean Field 
Games}

\author{Luis M. Brice\~no-Arias \& Julio Deride 
\& Sergio López-Rivera \& Francisco J. Silva}
\address{Departamento de Matem\'{a}tica, Universidad T\'{e}cnica Federico Santa Mar\'{i}a, Avenida Espa\~{n}a 1680, Valpara\'{i}so, Chile}
\email{luis.briceno@usm.cl, fernando.roldan@usm.cl}

\subjclass[2010]{47H05, 65K05, 65K15, 90C25, 90C90, 91-08.}

\begin{document}

\begin{abstract}
In this work, we study a constrained monotone inclusion involving 
the 
normal cone to a closed vector subspace and a priori information
on primal solutions. We model this 
information by imposing that solutions belong to the fixed 
point set of an averaged nonexpansive mapping. We characterize the 
solutions using an auxiliary inclusion that involves the partial inverse 
operator. Then, we propose the 
primal-dual partial inverse splitting and we prove its weak 
convergence to a solution of the inclusion, generalizing several 
methods in the 
literature. The 
efficiency of the proposed method is illustrated in 
multiple applications including constrained LASSO,  
stochastic arc capacity expansion problems in transport 
networks, and variational mean field games with non-local couplings.
\par
\bigskip

\noindent \textbf{Keywords.} {\it Constrained convex optimization, 
Constrained 
LASSO, Monotone operator 
theory, Partial inverse method, Primal-dual splitting,
 Stochastic arc capacity expansion,
Mean field games.}
\end{abstract}

\maketitle
\section{Introduction}
\label{sec1}	
In this paper, we propose a convergent algorithm for solving 
composite 
monotone inclusions 
involving a normal cone to a closed vector subspace and a priori 
information on the solutions. The precise formulation of the 
monotone inclusion under study is stated in 
Problem~\ref{problema1} below and
models several applications such as evolution 
inclusions \cite{ap1,Comi08,Sorin10}, variational inequalities 
\cite{19.Livre1,ap5}, 
partial differential equations (PDEs) \cite{ap2,ap3,ap4}, and various 
optimization problems. In the particular case when monotone 
operators are subdifferentials of proper lower semicontinuous 
convex functions, the inclusion covers the optimization problem 
with a priori information
\begin{equation}
\label{e:pricpintro}
\textrm{find} \,\, {x}\in S\cap \argminE_{x\in V} \big(F(x) + 
G(Lx) + H(x)\big),
\end{equation}
where  $S$ is a closed 
convex subset of a real Hilbert space $\HH$
modeling the a priori information on the solution, $V\subset\HH$ 
is a 
closed vector subspace, $L$ is a linear bounded 
operator from $\HH$ to a real Hilbert space $\GG$, 
$F\colon\HH\to\RX$, 
$G\colon\GG\to\RX$, 
and $H\colon\HH\to\RR$ are proper lower semicontinuous 
convex  functions, and $H$ is G\^ateaux 
differentiable. This class 
of problems appears in PDEs \cite[Section 3]{mercier}, 
signal and image processing \cite{aujol,clions,daub}, and stochastic 
traffic theory \cite{sheff,aplicacion_julio}, mean field games 
\cite{benamoucarlier15,BAKS},
among other fields. In the aforementioned applications, the 
vector subspace constraint models intrinsic properties of the 
solution or 
non-anticipativity in stochastic problems. In turn, the a priori 
information can be used to reinforce feasibility in the iterates, resulting 
in 
more efficient algorithms, as explored in \cite{jota1}.

In the case when the closed vector subspace is the whole Hilbert 
space, the 
monotone inclusion can be solved by the 
algorithm in \cite{jota1}. This method uses the a priori 
information represented by $S$ to 
improve the efficiency, generalizing the algorithm in 
\cite{11.vu} for monotone inclusions and in \cite{10.condat} for convex 
optimization.

In addition, when no a priori information is considered, 
the methods proposed in \cite{ref1,partial_inv} solve particular
instances of our problem using the partial inverse introduced in 
\cite{ref1}. This mathematical tool exploits the vector 
subspace structure of the inclusion and has been used, for example, 
in \cite{comb_partial_inverse,calc_inv_partial}.
Our problem in the more general context without a priori 
information can be solved by algorithms in 
\cite{11.vu,20.CombPes12,mon_skew}, using 
product space techniques without special consideration on the vector 
subspace structure. The product space formulation generates 
methods that include updates of high 
dimensional dual variables at each iteration, which reduce their 
performance.

The objective of this paper is to provide an algorithm for solving 
the inclusion under study
in its full generality, by taking advantage of 
the vector subspace structure and the a priori information of the 
inclusion. Our method is obtained from the combination of the 
algorithm in \cite{jota1} with partial inverse techniques developed in 
\cite{ref1,partial_inv,calc_inv_partial}. We illustrate the advantages of 
the partial inverse approach and the use of the a priori information by 
means of numerical experiences on particular instances of 
\eqref{e:pricpintro}. 
In this context, the a priori information is modeled by a set formed by 
some of the constraints of the problem and the additional projections 
in our method improve the speed of the convergence with respect to 
existing methods. In order to test the efficiency of our 
method, we consider the constrained LASSO problem, the arc 
capacity expansion in traffic networks, and Mean Field Games 
(MFG)
with non-local couplings.  The constrained LASSO problem 
combines sparse minimization and least squares under linear 
constraints, in which the vector subspace structure appears 
naturally. This problem appears in 
portfolio optimization, internet advertising, and curve estimation in 
statistics 
\cite{Gaines,James} 
and covers the generalized LASSO \cite{GenLASSO}, fused 
LASSO \cite{FusedLASSO}, and trend filtering on graphs \cite{Trend}, 
among others.
Our second application is devoted to solve a stochastic arc capacity 
expansion problem in transport networks.  We follow the two stage 
stochastic programming model described in 
\cite{aplicacion_julio,Transport1,Yin}, where an investment decision 
over arc capacity expansion is made in the first stage, and a traffic 
assignment problem under uncertain demand is solved in the second 
stage. Our solution strategy relies on a non-anticipativity approach, 
exploiting the vector subspace structure of the induced problem.
The last application is the numerical approximation of 
MFG equilibria 
\cite{Lasry_Lions_2007,Huang_et_all}. The theory of MFGs aims to 
describe equilibria of symmetric stochastic differential games with a 
continuum of agents. These, in turn, provide approximate equilibria for 
the corresponding game with a large, but finite, number of players. In 
its standard form, MFG equilibria is  characterized by a system of two 
coupled PDEs, which is called MFG system. In some particular cases, 
the MFG system is the first order optimality condition of a linearly 
constrained convex optimization problem
\cite{Lasry_Lions_2007,Benamou_et_al_2017} fitting in
our a priori information and vector subspace modeling (see 
\cite{BAKS}). 
In the context of  ergodic MFGs with non-local couplings 
\cite{Lasry_Lions_2006i,Lasry_Lions_2007,meszaros_silva_2018}, 
we illustrate the efficiency of our method and compare it with some 
benchmark algorithms. 		

The paper is organized as follows. In 
Section~\ref{preeliminares_paper2}, we set our notation 
and preliminaries. 
Section~\ref{prob_results} is devoted to the formulation our 
problem, the characterization of 
the solutions by using the 
partial inverse operator, the proof of the weak convergence 
of 
our algorithm, and a discussion of the connections with existing 
methods in the literature. In Section~\ref{tests12}, we implement 
the proposed algorithm in the context of constrained convex 
optimization. More precisely, in Section~\ref{sec:numer_LASSO} 
we compare the performance of our method with algorithms 
available in the literature for the constrained LASSO problem,
and in Section~\ref{sec_aplic}, we apply our method to 
the stochastic arc capacity expansion problem in a transport 
network. In Section~\ref{sec:numer_mfg}, we compare 
implementations to several formulations of the underlying 
optimization problem variational ergodic MFG with non-local 
couplings. Finally, we provide our conclusions and perspectives in 
Section~\ref{conclusion_paper2}.

\section{Notation and Background}
\label{preeliminares_paper2}
Throughout this paper, $\HH$ and $\GG$ are real Hilbert spaces. 
We denote their scalar products by $\left\langle \,\cdot\, | \,\cdot\, 
\right\rangle$ and their
associated norm by $\left\|\,\cdot\,\right\|$. The class of bounded 
linear operators from $\HH$ to a real Hilbert space $\GG$ is 
denoted 
by ${\mathcal{L}}(\HH,\GG)$ and if $\HH=\GG$ this class is denoted 
by 
${\mathcal{L}}(\HH)$. Given $L\in {\mathcal{L}}(\HH,\GG)$, its adjoint 
operator 
is 
denoted by $L^{*}\in {\mathcal{L}}(\GG,\HH)$. The projection operator 
onto 
a 
nonempty closed convex set $C\subset {\mathcal{H}}$ is denoted by 
$P_{C}$ 
and the normal cone to $C$
is denoted by $N_{C}$.
Let $A\colon\HH\rightrightarrows\HH$ be a set-valued 
operator. We denote by $\gr A$ its 
graph, by $A^{-1}\colon u\mapsto \menge{x\in\HH}{u\in Ax}$ its 
inverse operator, and
by $J_{A}:=(\Id+A)^{-1}$ its resolvent, where $\Id$ is the identity 
operator on $\HH$. Moreover, $A$ is $\rho$-strongly monotone for 
$\rho \geq 0$ iff, for every $(x,u)$ and $(y,v)$ in $\gr A$, we have 
$\scal{x-y}{u-v} \geq 
\rho\left\|x-y\right\|^{2}$, it is monotone iff it is $0-$strongly
monotone, and it is maximally monotone iff it is monotone and 
there is no monotone operator $B$ such that $\gr A\subsetneq 
\gr B$.
Let $V$ be a closed vector subspace of $\HH$. The partial inverse 
of $A$ with respect to $V$, denoted by 
$A_{V}$, is the operator defined via its graph by \cite{ref1}
\begin{equation}
\label{e:definvpart}
\gr A_{V}:=\{(x,u)\in\HH\times\HH\,:\,(P_{V}x + P_{V^{\perp}}u, 
P_{V}u + P_{V^{\perp}}x)\in \gr A\}.
\end{equation}
Let $T\colon \HH \rightarrow \HH$. The set of fixed points of $T$ is 
denoted by $ \Fix T$.
The operator $T$ is $\alpha-$averaged nonexpansive for 
some $\alpha\in \left]0,1\right[$ iff
\begin{equation}
(\forall (x,y)\in \HH^{2})\quad \|Tx-Ty\|^{2} \leq \|x-y\|^{2} 
-\left(\frac{1-\alpha}{\alpha}\right) \|(\Id-T)x-(\Id-T)y\|^{2}
\end{equation}
and it is 
$\beta-$cocoercive for some $\beta>0$ iff
\begin{equation}
(\forall (x,y)\in \HH^{2})\quad
\scal{ x-y}{Tx-Ty} \geq \beta 
\|Tx-Ty\|^{2}.
\end{equation}
The class of lower semicontinuous convex proper functions $f 
\colon\HH \rightarrow \left]-\infty, +\infty\right]$ is denoted by 
$\Gamma_{0}(\HH)$.
The 
subdifferential of $f\in \Gamma_{0}(\HH)$ is denoted by $\partial 
f$ and the proximity 
operator of $f$ on $x\in\HH$, denoted by $\prox_{f}x$, is 
the unique 
minimizer of $f+\|\cdot-x\|^2/2$. For every nonempty closed convex 
set $C\subset\HH$, we denote by $\iota_C\in\Gamma_0(\HH)$
the indicator function of $C$. 
For further information on convex analysis and monotone operator 
theory, the reader is referred to \cite{19.Livre1}.
\raggedbottom
\section{Problem and main result}
\label{prob_results}

We consider the following composite primal-dual inclusion problem
with a priori information.
\begin{problem}
\label{problema1}
Let $\HH$ and $\GG$ be real Hilbert spaces and let $V\subset 
\HH$ and 
$W\subset\GG$ be closed vector subspaces. Let $T\colon\HH\to\HH$ 
be an $\alpha-$averaged nonexpansive
operator for some $\alpha\in\left]0,1\right[$, let $L\in 
{\mathcal{L}}(\HH,\GG)$ be such 
that $\ran L \subset W$, let $A\colon \HH \rightrightarrows \HH$, 
$B\colon \GG \rightrightarrows \GG$, and $D\colon \GG 
\rightrightarrows \GG$ be maximally 
monotone operators, let $C\colon \HH \rightarrow \HH$ be a 
$\beta-$cocoercive operator for some $\beta\in\left]0,+\infty\right]$, 
and suppose that $D$ is 
$\delta-$strongly 
monotone for some $\delta\in\left]0,+\infty\right]$. 
The 
problem is to 
\begin{equation}
\label{eq1}\tag{$\mathcal{I}$}
\textrm{find} \quad (x,u)\in (V\cap \Fix T)\times W \quad \textrm{such 
	that} 
\quad \begin{cases}
0\in Ax + Cx +L^{*}u + N_{V}x\\
0\in B^{-1}u+D^{-1}u-Lx,
\end{cases}
\end{equation}
under the assumption that \eqref{eq1} admits solutions.
\end{problem}
Consider the case when  $W=\GG$, $A=\partial F$, $B=\partial G$, 
$C=\nabla H$,
and $D=\partial \ell$, where $F\in \Gamma_{0}(\HH)$, $G\in 
\Gamma_{0}(\GG)$, $H\colon \HH\rightarrow {\mathbb{R}}$ is a 
differentiable convex function with $\beta^{-1}-$Lipschitz gradient, and 
$\ell \in \Gamma_{0}(\GG)$ is $\delta-$strongly convex. By defining 
$G\infconv \ell$ as the infimal convolution of $G$ and $\ell$,
it follows from 
$\partial (G\infconv \ell)=((\partial G)^{-1}+(\partial \ell)^{-1})^{-1}$
\cite[Proposition~15.7(i) \& Proposition~25.32]{19.Livre1} that
Problem~\ref{problema1}  reduces to the constrained optimization 
problem
\begin{equation}
\label{e:pricp}
\textrm{find} \,\, {x}\in \Fix T\cap \argminE_{x\in V} \big(F(x) + 
(G\infconv \ell)(Lx) + H(x)\big),
\end{equation}
under standard qualification conditions. 
Note that from
\cite[Corollary~18.17]{19.Livre1}, $\nabla H$ is 
$\beta-$cocoercive, $\ell^*$ is G\^ateaux differentiable, and 
$\nabla\ell^*$ is $\delta-$cocoercive.
In the case when 
$T=P_C$ for some nonempty closed 
convex set $C$, \eqref{e:pricp} models optimization 
problems with a priori information on the solution 
\cite{jota1}. Moreover, \eqref{e:pricp} models the problem of finding
a common solution to two convex optimization problems when $T$ is 
such that $\Fix T=\arg\min 
\Phi$, for some convex function $\Phi$ (e.g., if $T=\prox_{\Phi}$, 
$\Fix T=\arg\min \Phi$). More generally, 
by using the a priori 
information on solutions with a suitable operator $T$, 
we can find common solutions to two 
monotone inclusions.

In the particular case when $V=\HH$, $W=\GG$, and $T=\Id$, 
\cite{11.vu} solves Problem~\ref{problema1}, and the corresponding 
optimization problem can be solved by the method 
proposed in \cite{10.condat}. Previous methods are generalizations 
of several classical splitting algorithms in particular instances, such as 
the proximal-point algorithm \cite{rock, martinet}, the 
forward-backward splitting \cite{lions}, and the Chambolle-Pock's 
algorithm \cite{12.CP}. In \cite{jota1}, previous methods are 
generalized for solving 
Problem~\ref{problema1} in the case when $W\subset\GG$ 
and $T$ is an arbitrary averaged nonexpansive operator.

In the case when $V\ne\HH$, the algorithms proposed in 
\cite{ref1,partial_inv} solve Problem~\ref{problema1} in particular 
instances, exploiting the vector subspace structure by using the partial 
inverse of a monotone operator. A convergent splitting method 
generalizing the partial inverse algorithm in \cite{ref1} is proposed in 
\cite{partial_inv} and solves Problem~\ref{problema1} when $B=0$. 
The algorithms proposed in \cite{11.vu,20.CombPes12,mon_skew} 
solves Problem~\ref{problema1} when $W=\GG$ and $T=\Id$, using 
product space techniques without special consideration on the vector 
subspace structure of $N_{V}$. The resulting method involves higher 
dimensional dual variables to be updated at each iteration, affecting 
the performance of the algorithm.

In the next section we provide our algorithm and 
main 
results.
\raggedbottom
\subsection{Algorithm and convergence}
\label{sec:mainres}
Our algorithm for solving 
Problem~\ref{problema1} is the following. 
\begin{algorith}
\label{algo:main}
Set $x^{0}\in V$, $\overline{x}^{0}=x^0$, $y^{0}\in V^{\perp}$, 
$u^{0}\in \GG$, $\tau >0$, and set $\gamma>0$. 
\begin{equation}
\label{eq5}
(\forall k\in {\mathbb{N}})\quad
\left\lfloor
\begin{array}{rl}
\eta^{k+1} &=J_{ \gamma 
	B^{-1}}\big(u^{k}+\gamma(L\overline{x}^{k} - 
D^{-1}u^{k})\big)\\
u^{k+1}&=P_{W}\eta^{k+1}\\
w^{k+1}&=J_{\tau A}\big(x^{k}+\tau 
y^{k}-\tau 
P_{V}(L^{*}u^{k+1} + 
Cx^{k})\big)\\
r^{k+1}&=P_{V}w^{k+1}\\
x^{k+1}&=P_{V}Tr^{k+1}\\
y^{k+1}&=y^{k}+{(r^{k+1}-w^{k+1})}/{\tau}\\\
\overline{x}^{k+1}&=x^{k+1}+r^{k+1}-x^{k}.
\end{array}
\right.
\end{equation}
\end{algorith}

Algorithm~\ref{algo:main} exploits the 
vector subspace structure via $P_V$ and $P_W$ and the a priori 
information on the solution of the inclusion via $T$. 

In the following theorem, we prove the weak convergence of the 
primal-dual sequences generated by Algorithm~\ref{algo:main} to a 
solution to Problem~\ref{problema1}.
We first characterize the solutions to 
Problem~\ref{problema1} as 
solutions to an auxiliary monotone inclusion involving the 
partial inverse of $A$ with respect to $V$ and we deduce 
the convergence of the iterates generated by 
Algorithm~\ref{algo:main} through 
a suitable application of the method in \cite{jota1} to 
the auxiliary inclusion.

\begin{theorem}
\label{teo1}
In the context of Problem \ref{problema1}, let 
$\big((x^k,u^k)\big)_{k\in\NN}$ be the sequence 
generated by Algorithm~\ref{algo:main}, where
we assume that 
$\tau \in \left]0,2\beta\right[$, $\gamma \in 
\left]0,2\delta\right[$, and that
\begin{equation}
\label{e:pramcond}
\|L\|^{2}< \left(\dfrac{1}{\tau} 
-\dfrac{1}{2\beta}\right)\left(\dfrac{1}{\gamma} 
-\dfrac{1}{2\delta}\right).
\end{equation}
Then $\big((x^k,u^k)\big)_{k\in\NN}$ converges weakly to 
a solution $(\widehat{x},\widehat{u})$ to 
Problem~\ref{problema1}.
\end{theorem}
\raggedbottom
\begin{proof}
We split the proof in two main parts. For the first part, the key point  
is to provide an equivalent formulation of Problem~\ref{problema1} 
which satisfies the 
hypotheses in \cite[Problem~3.1]{jota1} by using the properties of the 
partial inverse. In the second part, we prove that \eqref{eq5}
corresponds to the algorithm proposed in \cite[Theorem~3.1]{jota1}
applied to the equivalent formulation.

\begin{enumerate}
\item\label{dem_part_1}  Let $(x,u)\in(V\cap \Fix T)\times W$ be a 
solution to 
Problem~\ref{problema1}.
Note that 
\begin{align*}
B^{-1}+ D^{-1}
& = \left((\tau B)^{-1}+(\tau D)^{-1}\right)\circ (\tau\Id).
\end{align*}
\raggedbottom
Therefore, by using \eqref{e:definvpart}, 
there exists $y\in N_{V}x=V^{\perp}$ such that 
\begin{align}
\begin{cases}
y-L^{*}u-Cx \in Ax\\
Lx\in B^{-1}u+D^{-1}u
\end{cases}
&\Leftrightarrow\quad 
\begin{cases}
\tau(y-L^{*}u-Cx) \in (\tau A)x\\
Lx\in (\tau B)^{-1}(\tau u)+(\tau D)^{-1}(\tau u)
\end{cases}\label{e:eqaux}\\
&\Leftrightarrow\quad 
\begin{cases}
-\tau P_{V}(L^{*}u+Cx)
\in (\tau A)_{V}z\\
Lx\in (\tau B)^{-1}(\tau u)+(\tau D)^{-1}(\tau u)
\end{cases}\nonumber\\
\label{e:auxfin2}
&\Leftrightarrow\quad 
\begin{cases}
-P_{V}L^{*}v
\in (\tau A)_{V}z+\tau P_VCP_Vz\\
LP_Vz\in (\tau B)^{-1}v+(\tau D)^{-1}v,
\end{cases}
\end{align}
where
\begin{equation}
\label{sol_aux}
z:=x+\tau(y-P_{V^{\perp}}(L^{*}u+Cx)) \quad\text{and}\quad 
v:=\tau u\in W.
\end{equation}
In addition, since $P_Vz=x\in\Fix 
T\cap 
V$, then $TP_{V}z=P_{V}z$ and, thus, 
$P_{V}TP_{V}z=P_{V}^{2}z=P_{V}z=z-P_{V^{\perp}}z$, which yields 
$z\in\Fix M$, where
\begin{equation}
\label{op_M}
M=P_{V}TP_{V}+P_{V^{\perp}}.
\end{equation}
Thus, by setting 
\begin{equation}
\label{e:deftildes}
\begin{cases}
\widetilde{A}=(\tau A)_{V}\\
\widetilde{B}=\tau B\\
\widetilde{C}=\tau P_{V}CP_{V}\\
\widetilde{D}=\tau D\\
\widetilde L=LP_{V},
\end{cases}
\end{equation}
we conclude from \eqref{e:auxfin2} that $(z,v)$ 
defined in \eqref{sol_aux}
solves
\begin{equation}
\label{inc_inv_equiv}
\textrm{find}\quad (z,v)\in \Fix M\times W \quad \textrm{such 
	that}\quad
\left\lbrace
\begin{array}{ll}
-\widetilde{L}^{*}v\in \widetilde{A}z + \widetilde{C}z\\
\widetilde{L}z\in \widetilde{B}^{-1}v +  \widetilde{D}^{-1}v.
\end{array}
\right.
\end{equation}
In addition, we have that $\widetilde{A}$ and $\widetilde{B}$ are 
maximally monotone 
operators \cite[Proposition 20.44(v) \& 
Proposition~20.22]{19.Livre1},
$\widetilde{D}$ is $\tau\delta-$strongly monotone, and, since $P_{V}$ 
is a bounded linear operator,  
$\widetilde{L}$ is a bounded linear operator such that $\ran 
\widetilde{L} \subset \ran L \subset W$.
Moreover, since 
$P_{V}$ is nonexpansive and $P_{V}^{*}=P_{V}$, for 
every 
$(x,y)\in \HH \times \HH$ we have
\begin{align*}
\scal{\widetilde{C}x-\widetilde{C}y}{x-y}&=\tau 
\scal{CP_{V}x-CP_{V}y}{P_{V}x-P_{V}y}\nonumber\\
&\geq \tau\beta \|CP_{V}x-CP_{V}y\|^{2}\nonumber\\
&\geq \tau\beta \|P_{V}CP_{V}x-P_{V}CP_{V}y\|^{2}\nonumber\\
&= \dfrac{\beta}{\tau}\|\widetilde{C}x-\widetilde{C}y\|^{2},
\end{align*}
which implies that $\widetilde{C}$ is 
$\frac{\beta}{\tau}-$cocoercive. Furthemore, we claim that $M$ 
defined in \eqref{op_M} is $\alpha-$averaged nonexpansive. Indeed, 
let $(x,z)\in 
\HH^{2}$. Since $T$ is $\alpha-$averaged nonexpansive, $\Id = 
P_{V}+P_{V^{\perp}}$, $P_{V}$ is nonexpansive, and 
$\Id-M=P_V-P_V\circ T\circ 
P_V=P_V\circ (\Id-T)\circ P_V$, we 
have
\begin{align*}
\|Mx-Mz\|^{2}&=\|P_{V}TP_{V}x+P_{V^{\perp}}x 
-P_{V}TP_{V}z-P_{V^{\perp}}z\|^{2}\nonumber\\
&=\|P_{V}TP_{V}x-P_{V}TP_{V}z\|^{2} + 
\|P_{V^{\perp}}x-P_{V^{\perp}}z\|^{2}\nonumber\\
&\leq \|TP_{V}x-TP_{V}z\|^{2} + 
\|P_{V^{\perp}}x-P_{V^{\perp}}z\|^{2}\nonumber\\
&\leq  
\|P_{V}x-P_{V}z\|^{2}-\left(\dfrac{1-\alpha}{\alpha}\right)\|(\Id-T)P_{V}x-(\Id-T)P_{V}z\|^{2}
\nonumber\\
&\quad+ \|P_{V^{\perp}}x-P_{V^{\perp}}z\|^{2}\nonumber\\
&\leq \|x-z\|^{2} - 
\left(\dfrac{1-\alpha}{\alpha}\right)\|P_{V}(\Id-T)P_{V}x - 
P_{V}(\Id-T)P_{V}z\|^{2}\nonumber\\
&= \|x-z\|^{2} - \left(\dfrac{1-\alpha}{\alpha}\right)
\|(\Id-M)x-(\Id-M)z\|^{2}.
\end{align*}

Altogether, we conclude that \eqref{inc_inv_equiv} reduces to 
\cite[Problem~3.1]{jota1}.

Conversely, suppose that $(\widehat{z},\widehat{v}) \in \Fix M \times 
W$ 
solves \eqref{inc_inv_equiv}. By setting 
$\widehat{x}:=P_V\widehat{z}\in V$,
$\widehat{u}:=\widehat{v}/\tau\in W$, and 
$\widehat{y}:=P_{V^{\perp}}(\frac{1}{\tau}\widehat{z} +
L^{*}\widehat{u}+CP_{V}\widehat{z}) \in V^{\perp}$,
we get that
\begin{align*}
\widehat{z}&=P_{V}\widehat{z}+P_{V^{\perp}}\widehat{z}\\
&=\widehat{x}+\tau P_{V^{\perp}}\left(\frac{1}{\tau}\widehat{z}\right)\\
&=\widehat{x}+\tau\left(P_{V^{\perp}}\left(\frac{1}{\tau}\widehat{z} +
L^{*}\widehat{u}+CP_{V}\widehat{z}\right)-P_{V^{\perp}}(L^{*}\widehat{u}+C\widehat{x})\right)\\
&=\widehat{x}+\tau(\widehat{y}-P_{V^{\perp}}(L^{*}\widehat{u}+C\widehat{x})).
\end{align*}
Hence, $(\widehat{x},\widehat{u},\widehat{y},\widehat{z},\widehat{v})$ 
satisfies \eqref{sol_aux}. Then,
since $(\widehat{z},\widehat{v})$ satisfies \eqref{e:auxfin2}, we 
deduce that $(\widehat{x},\widehat{u},\widehat{y})$ satisfies 
\eqref{e:eqaux}. In addition, since $\widehat{z}\in \Fix M$, then 
$\widehat{z}=P_{V}TP_{V}\widehat{z}+\widehat{z}-P_{V}\widehat{z}$ 
and therefore $\widehat{x}\in\Fix(P_{V}T)$.
Now, since Problem~\ref{problema1} has 
solutions, $\Fix P_{V}\cap \Fix T =V\cap \Fix T\neq\varnothing$. 
Then, by \cite[Proposition~4.49(i)]{19.Livre1}, it follows that $\Fix 
(P_{V}T)= V\cap \Fix T$ and $\widehat{x}\in V\cap \Fix T$.

That is,
$(\widehat{x},\widehat{u})\in(V\cap\Fix T)\times W$ and therefore 
$(\widehat{x},\widehat{u})$ is a solution 
to \eqref{eq1}.	

\item\label{dem_part_2}

For every $k\in {\mathbb{N}}$, define $z^{k}:=x^{k}+\tau 
y^{k}$, 
$p^{k}:=r^{k}+\tau y^{k}$, 
$\overline{z}^{k+1}:={z}^{k+1}+{p}^{k+1}-z^{k}$, 
$\overline{z}^{0}:=z^{0}$, and
$\widetilde{z}^k=z^k-\tau P_V(L^*u^{k+1}+Cx^k)$. Note that  
\eqref{eq5} yields
$\{r^{k}\}_{k\in \mathbb{N}} 
\subset V$, $\{x^{k}\}_{k\in \mathbb{N}} \subset V$, and 
$\{y^{k}\}_{k\in \mathbb{N}} \subset V^{\perp}$. Hence, since 
for every $k\in\NN$, 
$r^k=P_Vp^k$ and $x^k=P_Vz^k$, it follows 
from \eqref{eq5}  that, for every $k\in\NN$,
\begin{equation}
\label{e:defz}
\left\lbrace
\begin{array}{ll}
z^{k+1}=P_{V}Tr^{k+1}+\tau 
y^{k+1}=P_{V}TP_{V}p^{k+1}+P_{V^{\perp}}p^{k+1}
=Mp^{k+1}\\
P_{V}\overline{z}^{k+1}=x^{k+1}+r^{k+1}-x^{k}=\overline{x}^{k+1}.\\
\end{array}
\right.
\end{equation}

In addition, 
$P_{V}\overline{z}^{0}=P_{V}z^{0}=x^{0}=\overline{x}^{0}$. Now,
from \eqref{eq5} and \cite[Proposition~3.1(i)]{calc_inv_partial} we 
deduce 
\begin{align}
{p}^{k+1}&=r^{k+1}+\tau y^{k+1}\nonumber\\
&=2r^{k+1}-w^{k+1}+\tau y^{k}\nonumber\\
&=(2P_{V}-\Id)J_{\tau 
	A}\widetilde{z}^{k}+P_{V^{\perp}}\widetilde{z}^{k}\nonumber\\
&=J_{(\tau A)_{V}}\widetilde{z}^{k}.
\end{align}


Also, note that
\begin{align*}
\tau J_{\gamma{B}^{-1}} &= 
\left((\Id+\gamma 
B^{-1})\circ(\tau^{-1}\Id)\right)^{-1}\nonumber\\
&=\left(\tau^{-1}(\Id+\tau\gamma(\tau 
B)^{-1})\right)^{-1}\nonumber\\
&=J_{\tau\gamma(\tau B)^{-1}}\circ(\tau\Id).
\end{align*}
Let us define  
$\sigma=\tau\gamma\in \left]0,2\tau\delta\right[$ and, for every 
$k\in 
{\mathbb{N}}$, set $v^{k}=\tau u^{k}$ and 
$\zeta^{k}=\tau\eta^{k}$. 
Thus, from $P_{V}^{*}=P_{V}$ and the linearity of $P_{V}$ and 
$P_{W}$, we deduce from \eqref{e:deftildes} and \eqref{e:defz} that 
\eqref{eq5} reduces to
\begin{equation}
\label{aux_alg}
(\forall k\in {\mathbb{N}})\quad
\left\lfloor
\begin{array}{ll}
\zeta^{k+1} = 
J_{\sigma{\widetilde{B}}^{-1}}(v^{k}+
\sigma(\widetilde{L}\overline{z}^{k}
- {\widetilde{D}}^{-1}v^{k}))\\
v^{k+1}=P_{W}\zeta^{k+1}\\
p^{k+1}=J_{\widetilde{A}}(z^{k}-({\widetilde{L}}^{*}v^{k+1} + 
\widetilde{C}z^{k}))\\
z^{k+1}=Mp^{k+1}\\
\overline{z}^{k+1}=z^{k+1}+p^{k+1}-z^{k}.
\end{array}
\right.
\end{equation}
On the other hand, from \eqref{e:pramcond} we obtain
\begin{equation}
\|\widetilde{L}\|^{2} \leq \|L\|^{2} < 
\left(1-\dfrac{1}{\frac{2\beta}{\tau}}\right)
\left(\dfrac{1}{\sigma}-\dfrac{1}{2\tau\delta}\right)
\end{equation} 
and $1\in \left]0,2\beta/\tau\right[$.
Altogether,
it follows that \eqref{aux_alg} is a particular case of the algorithm in 
\cite[Theorem~3.1]{jota1}. In addition, from part \ref{dem_part_1}, we 
have that problem \eqref{inc_inv_equiv} satisfies the hypothesis of 
\cite[Problem~3.1]{jota1}. Then, from \cite[Theorem~3.1(ii)]{jota1}
there exists $(\widehat{z},\widehat{v}) \in \Fix M \times W$ 
solution 
to \eqref{inc_inv_equiv} such that $(z^{k},v^{k})\rightharpoonup 
(\widehat{z},\widehat{v})$. Therefore, $u^{k}=v^{k}/\tau 
\rightharpoonup 
\widehat{v}/\tau=:\widehat{u}$ and, since $P_V$ is weakly continuous, 
we have $x^{k}=P_{V}z^{k} 
\rightharpoonup P_{V}\widehat{z}=:\widehat{x}$.
Furthermore, from part \ref{dem_part_1}, we conclude that 
$(\widehat{x},\widehat{u})\in (V\cap\Fix T)\times W$ is solution to 
\eqref{eq1}.

\end{enumerate}
\end{proof}

\begin{remark}
\label{obs_teo}
\begin{enumerate}
	\item\label{obs_1} When $T$ is weakly continuous, we have 
	$Tr^{k}\rightharpoonup 
	\widehat{x}$, where $(r^k)_{k\in\NN}$ is defined in \eqref{eq5}. 
	Indeed, by the proof of \cite[Theorem~3.1(ii)]{jota1}, the 
	sequence $(p^{k})_{k\in {\mathbb{N}}}$ in the algorithm 
	\eqref{aux_alg} satisfies that $p^{k}-z^{k}\rightarrow 0$. Then 
	$p^{k}=p^{k}-z^{k}+z^{k} \rightharpoonup \widehat{z}$ and, 
	since $P_{V}$ is weakly continuous, it follows that 
	$r^{k}=P_{V}p^{k}\rightharpoonup 
	P_{V}\widehat{z}=\widehat{x}$. Thus, since $\widehat{x}\in \Fix 
	T$, $Tr^{k}\rightharpoonup T\widehat{x}=\widehat{x}$.
	This fact helps to obtain a faster convergence in the context of 
	convex optimization with affine linear constraints, as we will see 
	in 	our numerical experiences. Indeed, in this 
	case we take $T=P_{C}$, where $C$ 
	represents some selection of the affine linear constraints ($T$ is 
	weakly continuous by \cite[Proposition 4.19(i)]{19.Livre1}). 
	Therefore, $\{Tr^{k}\}_{k\in\mathbb{N}}\subset C$, which 
	impose feasibility of the converging iterates to the primal solution. 

	In particular, if $\HH$ is finite dimensional and $T=P_{C}$ for 
	a nonempty closed convex $C\subset \HH$, we deduce, from 
	Remark~\ref{obs_teo}\eqref{obs_1} and the fact 
	that $P_{C}$ is continuous, that
	$Tr^{k}\rightarrow 
	\widehat{x}$.
	\item When $V=\HH$, we have that $V^{\perp}=\{0\}$ 
	and $P_{V}=\Id$. Thus, the algorithm \eqref{eq5} reduces to
	\begin{equation*}
	(\forall k\in {\mathbb{N}})\quad
	\left\lfloor
	\begin{array}{ll}
	\eta^{k+1} = J_{\gamma B^{-1}} (u^{k}+\gamma(L\overline{x}^{k} - 
	D^{-1}u^{k}))\\
	u^{k+1}=P_{W}\eta^{k+1}\\
	w^{k+1}=J_{\tau A}(x^{k}-\tau(L^{*}u^{k+1} + Cx^{k}))\\
	x^{k+1}=Tw^{k+1}\\
	\overline{x}^{k+1}=x^{k+1}+w^{k+1}-x^{k},
	\end{array}
	\right.
	\end{equation*}
	which is the algorithm proposed in \cite[Theorem~3.1]{jota1} when 
	the stepsizes $\tau$ and $\gamma$ are fixed. 
	\item When $T=\Id$, $W=\GG$, and 
	$B=D^{-1}=0$, we have that for every $\lambda>0$,
	$J_{\lambda B}=\Id$. Hence $J_{\gamma B^{-1}} = 0$ by 
	\cite[Proposition~23.7(ii)]{19.Livre1}. Then, the algorithm 
	\eqref{eq5} reduces 
	to
	\begin{equation*}
	(\forall k\in {\mathbb{N}})\quad
	\left\lfloor
	\begin{array}{ll}
	\widetilde{z}^{k}=x^{k}+\tau y^{k}-\tau P_{V} Cx^{k}\\
	w^{k+1}=J_{\tau A}\widetilde{z}^{k}\\
	x^{k+1}=P_{V}w^{k+1}\\
	y^{k+1}=y^{k}+{(x^{k+1}-w^{k+1})}/{\tau},
	\end{array}
	\right.
	\end{equation*}
	which is the algorithm proposed in \cite[Corollary~5.3]{partial_inv} 
	without relaxation steps ($\lambda_{n}\equiv 1$).
	\item In the context of the convex optimization problem 
	\eqref{e:pricp}, the proposed algorithm \eqref{eq5} reduces to
	\begin{equation}
	\label{alg_opt}
	(\forall k\in {\mathbb{N}})\quad
	\left\lfloor
	\begin{array}{ll}
	\eta^{k+1} = \prox_{\gamma 
		G^*}\big(u^{k}+\gamma(L\overline{x}^{k} - 
	\nabla 
	\ell^{*}(u^{k}))\big)\\
	u^{k+1}=P_{W}\eta^{k+1}\\
	\widetilde{z}^{k}=x^{k}+\tau y^{k}-\tau P_{V}(L^{*}u^{k+1} + \nabla 
	H(x^{k}))\\
	w^{k+1}=\prox_{\tau F}\widetilde{z}^{k}\\
	r^{k+1}=P_{V}w^{k+1}\\
	x^{k+1}=P_{V}Tr^{k+1}\\
	y^{k+1}=y^{k}+{(r^{k+1}-w^{k+1})}/{\tau}\\\
	\overline{x}^{k+1}=x^{k+1}+r^{k+1}-x^{k}.
	\end{array}
	\right.
	\end{equation}
	In particular, when $T=\Id$, $V=\HH$, $W=\GG$, and 
	$\ell=\iota_{\{0\}}$, we deduce that the algorithm 
	\eqref{alg_opt} is equivalent to
	\begin{equation}
	\label{alg_condat}
	(\forall k\in {\mathbb{N}})\quad
	\left\lfloor
	\begin{array}{ll}
	u^{k+1} = \prox_{\gamma G^{*}}\big(u^{k}+\gamma 
	L\overline{x}^{k}\big)\\
	\widetilde{z}^{k}=x^{k}-\tau(L^{*}u^{k+1} + \nabla H(x^{k}))\\
	x^{k+1}=\prox_{\tau F}\widetilde{z}^{k}\\
	\overline{x}^{k+1}=2x^{k+1}-x^{k},
	\end{array}
	\right.
	\end{equation}
	which is an error-free version of the algorithm proposed in
	\cite[Algorithm~3.1]{10.condat}. If  
	additionally $H=0$, the method \eqref{alg_condat} reduces to 
	\cite[Algorithm~1]{12.CP}.
\end{enumerate}
\end{remark} 

\section{Applications and Numerical Experiences}
\label{tests12}
In this section, we illustrate the efficiency of the proposed method in 
three applications.
First, we consider a sparse constrained 
convex 
optimization problem without including a priori 
information on the solution ($T=\Id$), called constrained LASSO 
\cite{Gaines,James}. The constraint is given by the kernel of a 
linear operator and we apply our primal-dual method
exploiting the vector subspace structure of the problem.
The second numerical experiment is the application 
of the proposed method to the arc capacity expansion problem in 
transport networks \cite{aplicacion_julio}. The 
problem is to find the optimal investment decision in arc capacity 
and network flow 
operation under an uncertain environment.
We solve the 
two-stage stochastic arc capacity expansion problem over a 
directed graph using our primal-dual partial inverse method 
in which the closed vector subspace includes the 
non-anticipativity constraint.  In our last application, 
we consider the finite
difference approximation introduced in 
\cite{Achdou_Capuzzo_Dolcetta_2010} of a second order ergodic 
variational MFGs with non-local couplings. This 
discretization preserves the variational structure of the resulting 
system of equations, which is solved using our primal-dual partial 
inverse method. In this framework, the underlying vector subspace 
includes the linear discrete ergodic Fokker-Planck 
equation appearing in the associated optimization problem.

All proposed algorithms are implemented in MATLAB 2020A and run 
in a 
computer with MacOS 11.6, 3 GHz 6-Core Intel Core i5 8GB RAM. 
\subsection{Constrained LASSO}
\label{sec:numer_LASSO}
We consider the following problem
\begin{equation}
\label{prob_num_exp}
\mini_{\substack{x\in \RR^{n}\\ Rx=0}}\,\, \alpha\|x\|_{1} + 
\dfrac{1}{2} \|Ax-b\|_{2}^{2},
\end{equation}
where $\alpha>0$, $A\in \RR^{p\times n}$, $R\in \RR^{m\times n}$ 
satisfies $\ker R^{\top} =\{0\}$, and $b\in \RR^{p}$. Note that, by 
setting $f=\alpha \|\cdot\|_{1}$, the 
problem in \eqref{prob_num_exp} can be written in at least the 
following three 
equivalent manners:
\begin{equation}
\label{prob_num_exp3}
\mini_{x\in \RR^{n}} f(x) + \dfrac{1}{2}g_{2}(Ax) + 
\dfrac{1}{2}\iota_{\{0\}}(Rx),\:\:\text{where}\:\:	g_{2}=\|\cdot - 
b\|_{2}^{2};
\end{equation}
\begin{equation}
\label{prob_num_exp2}
\mini_{x\in \ker R} f(x) + h(x),\:\:\text{where}\:\:h=\dfrac{1}{2} 
\|A(\cdot)-b\|_{2}^{2},
\end{equation}
and
\begin{equation}
\label{prob_num_exp1}
\mini_{x\in \ker R} f(x) + g_{1}(Ax),\:\:\text{where}\:\: 
g_{1}=\dfrac{1}{2}\|\cdot - b\|_{2}^{2}.
\end{equation}
Observe that $f\in \Gamma_{0}(\RR^{n})$, $g_{1}\in 
\Gamma_{0}(\RR^{p})$,  and 
$g_{2}=2g_{1}\in\Gamma_{0}(\RR^{p})$. 
Therefore,
the problem in  \eqref{prob_num_exp3} satisfies the hypotheses in 
\cite[Corollary~4.2(i)]{11.vu}. Thus, since 
\cite[Proposition~24.8(i)\& Theorem~14.3(ii)]{19.Livre1} yields ($\forall 
\gamma>0$)\,\, 
$\prox_{\gamma 
g_{2}^{*}}\colon x\mapsto 2(x-\gamma b)/(\gamma+2)$,
the primal-dual method proposed in 
\cite[Corollary~4.2(i)]{11.vu} reduces to
\begin{equation}
\label{alg_exp3}
(\forall k\in {\mathbb{N}})\quad
\left\lfloor
\begin{array}{ll}
x^{k+1} = \prox_{\tau\alpha 
\|\cdot\|_{1}}(x^{k}-\frac{\tau}{2}A^{\top}v_{1}^{k} - 
\frac{\tau}{2}R^{\top}v_{2}^{k})\\
y^{k}=2x^{k+1}-x^{k}\\
v_{1}^{k+1}=2(v_{1}^{k} + \sigma_{1} Ay^{k}-\sigma_{1} 
b)/(\sigma_{1}+2)\\
v_{2}^{k+1}=v_{2}^{k}+\sigma_{2}Ry^{k},
\end{array}
\right.
\end{equation}
where $(x^{0},v_{1}^{0},v_{2}^{0})\in 
\RR^{n}\times\RR^{p}\times\RR^{m}$ and the strictly positive constants
$\tau$, $\sigma_{1}$, and $\sigma_{2}$ satisfy the condition
$\sqrt{\frac{\tau}{2}\sigma_{1}\|A\|^{2} 
+\frac{\tau}{2}\sigma_{2}\|R\|^{2}}<1$.
On the other hand, we have that $h\colon 
\RR^{n} \rightarrow \RR$ is a differentiable convex function with 
$\|A^{\top}A\|-$ Lipschitz continuous gradient.
Then, by setting $V=\ker R$, the problem in 
\eqref{prob_num_exp2} can be solved by the algorithm in
\cite[Proposition~6.7]{partial_inv}, which reduces to
\begin{equation}
\label{alg_exp2}
(\forall k\in {\mathbb{N}})\quad
\left\lfloor
\begin{array}{ll}
w^{k+1} = \prox_{\lambda\alpha \|\cdot\|_{1}} (x^{k}+\lambda 
y^{k}-\lambda P_{\ker R}\nabla h(x^{k}))\\
x^{k+1}=P_{\ker R}w^{k+1}\\
y^{k+1}=y^{k}+(x^{k+1}-w^{k+1})/{\lambda},
\end{array}
\right.
\end{equation}
where $x^{0}\in \ker R$, $y^{0}\in(\ker R)^{\perp}$, and 
$\lambda\in\left]0,2/\|A^{\top}A\|\right[$. 
Moreover, from \cite[Proposition~24.8(i)\& 
Theorem~14.3(ii)]{19.Livre1} we have ($\forall \gamma>0$)\,\, 
$\prox_{\gamma 
g_{1}^{*}}\colon x\mapsto (x-\gamma b)/(\gamma+1)$. By setting 
$\HH=\RR^n$, $\GG=\RR^p$, and $V=\ker R$, 
the problem in 
\eqref{prob_num_exp1} 
satisfies the 
hypotheses in \eqref{e:pricp}, which is a particular instance of 
Problem~\ref{problema1}. Therefore,  \eqref{alg_opt}  in the case 
$T=\Id$ reduces to
\begin{equation}
\label{alg_exp1}
(\forall k\in {\mathbb{N}})\quad
\left\lfloor
\begin{array}{ll}
u^{k+1} = (u^{k} + \gamma (A\overline{x}^{k}-b))/(\gamma+1) \\
\widetilde{z}^{k}=x^{k}+\tau y^{k}-\tau P_{\ker R}A^{\top}u^{k+1}\\
w^{k+1}=\prox_{\tau\alpha \|\cdot\|_{1}}\widetilde{z}^{k}\\
x^{k+1}=P_{\ker R}{w}^{k+1}\\
y^{k+1}=y^{k}+(x^{k+1}-w^{k+1})/{\tau}\\
\overline{x}^{k+1}=2x^{k+1}-x^{k},
\end{array}
\right.
\end{equation}
where $x^{0}\in\ker R$, $\overline{x}^{0}=x^{0}$, $y^{0}\in 
(\ker R)^{\bot}$, $u^{0}\in \RR^{p}$, and $(\tau,\gamma)\in 
\left]0,+\infty\right[^{2}$ are such that $\tau\gamma\|A\|^{2}<1$. 

Note that, since $\ker R^{\top}=\{0\}$, $RR^{\top}$ is invertible. 
Then, by \cite[Example 29.17(iii)]{19.Livre1}, we have $P_{\ker 
R}=\Id-R^{\top}(RR^{\top})^{-1}R$. On the other hand, by 
\cite[Proposition~24.11 \& Example~24.20]{19.Livre1}, we have
\begin{equation*}
(\forall \tau>0)\quad \prox_{\tau 
\|\cdot\|_{1}}\colon x \mapsto  
\big(
\operatorname{sign}(x_i) \max \{|x_i|-\tau, 0\}
\big)_{1\le i\le n},
\end{equation*}
where $\operatorname{sign}$ is $1$ when the argument is 
positive
and $-1$ if it is strictly negative.

For each method, we obtain the average execution time and the 
average number of iterations from $20$ random instances for the 
matrices $A$, $R$, and  $b$, using $\alpha=1$. 
We measure the efficiency for different values of $m$, $n$, and 
$p$. We label the algorithm in \eqref{alg_exp3} as \textit{PD 
generalized}, algorithm in \eqref{alg_exp2} as \textit{FB with 
subspaces}, and algorithm in \eqref{alg_exp1} as \textit{PD with 
subspaces}. For every algorithm, we obtain the  values of $\tau$, 
$\gamma$, $\lambda$, $\sigma_{1}$, and $\sigma_{2}$ 
by discretizing the parameter set in which the algorithm 
converges and selecting the parameters such that the 
method stops in a 
minimum number of iterations. This procedure is repeated for every 
dimension of matrices and vectors. In particular, we fix
$\sigma_{1}=\sigma_{2}$ for the method in \eqref{alg_exp3}.
The results are shown in Table~\ref{tab:1}.

\begin{table}[H]
\caption{Average execution time (number of iterations) with relative 
	error tolerance $e=10^{-6}$}
\centering
\begin{tabular}{c|c|c|c|}
	\cline{1-4}
	\multicolumn{1}{|c|}{$(n,p,m)$}        & PD with subspaces & FB with 
	subspaces & PD generalized \\ \hline
	\multicolumn{1}{|c|}{$(500,250,25)$}   &      0.639 (3284)       
	&3.909 (21059)
	&   1.317 (4469)      \\ 
	\multicolumn{1}{|c|}{$(500,250,50)$}  &  0.822 (3565)
	& 5.014 (22145)
	&   1.254 (5081)
	\\ 
	\multicolumn{1}{|c|}{$(500,250,100)$}  &  1.289 (3523)
	& 7.956 (21374)
	&  1.817 (5527)
	\\ \hline
	\multicolumn{1}{|c|}{$(500,750,25)$}  &     0.488 (2184)
	& 3.615 (16577)
	&  1.012 (2991)
	\\ 
	\multicolumn{1}{|c|}{$(500,750,50)$} & 0.579 (2229)
	&  4.197 (16445)
	&   0.862 (3063)
	\\ 
	\multicolumn{1}{|c|}{$(500,750,100)$} & 0.854 (2117)
	&  6.345 (15853)
	&    1.129 (3066)
	\\ \hline
	\multicolumn{1}{|c|}{$(1000,500,50)$}   &    3.032 (8910)        
	&  	16.781 (49199)
	&     4.665 (11125)    \\ 
	\multicolumn{1}{|c|}{$(1000,500,100)$}  &  5.278 (9716)
	& 27.937 (51287)
	&  5.591 (12615)
	\\ 
	\multicolumn{1}{|c|}{$(1000,500,200)$}  &      10.830 (9036) 
	& 57.976 (48314)
	&  7.283 (13014)
	\\ \hline
	\multicolumn{1}{|c|}{$(1000,1500,50)$}  &   6.252 (4869)
	& 44.335 (34553)
	&    7.610 (6378)
	\\ 
	\multicolumn{1}{|c|}{$(1000,1500,100)$} & 7.911 (4992)
	&  54.217 (34507)
	&    8.691 (6484)
	\\ 
	\multicolumn{1}{|c|}{$(1000,1500,200)$} & 11.570 (4642)
	& 79.844 (32110)
	&   9.882 (6169)
	\\ \hline
\end{tabular}
\label{tab:1}
\end{table}
We observe a substantial gain in efficiency when we use 
\textit{PD with 
subspaces} with respect to the other two methods. The number of 
iterations is reduced in $25-30\%$ with respect to \textit{PD 
generalized}. The construction of \textit{PD with 
subspaces} exploiting the vector subspace and primal-dual structure 
of the problem explains these benefits.
However, the computational time used by \textit{PD with 
subspaces} is larger than that of \textit{PD generalized} when the 
vector subspace is smaller ($m=200$). The presence of two 
projections onto $\ker R$ at each iteration of 
\textit{PD with subspaces} explains this behaviour since the matrices 
to be inverted are of larger dimension.

\subsection{Capacity Expansion Problem in Transport 
Networks}
\label{sec_aplic}
In this section we aim at solving the traffic assignment problem with 
arc-capacity expansion with minimal cost on a network 
under uncertainty. Let  ${\mathcal{A}}$ be the set of 
arcs and let ${\mathcal{O}}$ and ${\mathcal{D}}$ be the sets of origin 
and 
destination nodes of the network, respectively. The set of 
routes from 
$o\in {\mathcal{O}}$ to $d\in {\mathcal{D}}$ is denoted by $R_{od}$ 
and
$R:=\cup_{(o,d)\in {\mathcal{O}}\times{\mathcal{D}}} 
R_{od}$ is the set of all routes. The arc-route 
incidence matrix $N\in 
{\mathbb{R}}^{|{\mathcal{A}}|\times |R|}$ is defined by $N_{ar}:=1$,
if arc $a$ belongs to the 
route $r$, and $N_{ar}:=0$, otherwise.

The uncertainty is modeled by a finite set $\Xi$ of
possible scenarios. For every scenario $\xi\in\Xi$, $p_{\xi}\in 
[0,1]$ is its probability of occurrence, 
$h_{od,\xi}\in {\mathbb{R}_+}$ is 
the forecasted demand from $o\in {\mathcal{O}}$ to $d\in 
{\mathcal{D}}$, 
$c_{a,\xi}\in {\mathbb{R}}_{+}$ 
is the corresponding capacity of the arc $a\in {\mathcal{A}}$,
$t_{a,\xi}\colon 
{\mathbb{R}}_{+}\rightarrow {\mathbb{R}}_{+}$ is an 
increasing and $\chi_{a,\xi}-$Lipschitz continuous travel time 
function on arc $a\in {\mathcal{A}}$, for some $\chi_{a,\xi}>0$, 
and the variable $f_{r,\xi}\in 
{\mathbb{R}_+}$ 
stands for the flow in route $r\in R$. 

In the problem of this section, we consider the expansion of 
flow capacity at each arc in order to improve the efficiency of the 
network operation. We model this decision 
making process in a two-stage stochastic problem. The first stage 
reflects the 
investment in capacity and the second corresponds to the 
operation of the network in an uncertain environment. 

In order to solve this problem, we take a non-anticipativity approach 
\cite{aplicacion_julio}, 
letting our first stage decision variable depend on the scenario
and	imposing a non-anticipativity constraint. We denote by 
$x_{a,\xi}\in\RR_+$ the variable of capacity expansion on arc 
$a\in\mathcal{A}$ in scenario $\xi\in\Xi$ and the non-anticipativity 
condition is defined by the constraint 
\begin{equation*}
\mathcal{N}:=\big\{x\in\RR^{|\mathcal{A}|\times|\Xi|}\,:\,(\forall 
(\xi,\xi')\in 
\Xi^{2})\quad x_{\xi}=x_{\xi'}\big\},
\end{equation*}
where $x_{\xi}\in\RR^{|\mathcal{A}|}$ is
the vector of capacity expansion for scenario $\xi\in\Xi$ and we denote
$f_{\xi}\in\RR^{|R|}$ analogously.
We restrict the capacity expansion variables by imposing, for every 
$a\in \mathcal{A} $ and $\xi\in\Xi$,
$x_{a,\xi}\in\left[0,M_a\right]$, where $M_a>0$ represents the upper 
bound of 
capacity expansion on arc $a\in \mathcal{A}$. Additionally,
we model the expansion investment cost via a quadratic function
defined by a symmetric positive definite matrix
$Q\in \RR^{|{\mathcal{A}}|\times |{\mathcal{A}}|}$.

\begin{problem}
\label{prob_transporte}
The problem is to
\begin{equation*}
\mini_{(x,f)\in  (\mathcal{N}\cap D^{|\Xi|})\times 
	{\mathbb{R}}_{+}^{|R||\Xi|}}  \displaystyle\sum_{\xi\in \Xi} 
p_{\xi}\left[\displaystyle\sum_{a\in {\mathcal{A}}} 
\displaystyle\int_{0}^{(Nf_{\xi})_a} 
t_{a,\xi}(z)dz+\frac{1}{2}x_{\xi}^{\top}Qx_{\xi}\right]
\end{equation*}
\begin{align}
\text{s.t.}				\hspace{1cm} (\forall \xi\in \Xi)(\forall a\in 
{\mathcal{A}})\hspace{1cm}
(Nf_{\xi})_a 
-x_{a,\xi}&\leq c_{a,\xi},
\label{rest2_prob_tr}\\
(\forall \xi\in \Xi)(\forall (o,d)\in 
{\mathcal{O}}\times{\mathcal{D}})\hspace{1cm}
\displaystyle\sum_{r\in R_{od}} f_{r,\xi} &=h_{od,\xi},
\label{rest1_prob_tr}
\end{align}
where $D:=\bigtimes_{a\in\mathcal{A}}[0,M_a]$, and we assume 
the 
existence of solutions.
\end{problem}
The first term of the objective function in 
Problem~\ref{prob_transporte} 
represents the expected operational cost of the network and only 
depends on traffic flows. This term is leads to optimality conditions
defining a Wardrop equilibrium \cite{beck}.
 The second term in the objective function is the 
expansion investment cost. Constraints in \eqref{rest2_prob_tr} 
represent that, for every arc $a\in\mathcal{A}$, the flow cannot 
exceed the expanded 
capacity $c_{a,\xi}+x_{a,\xi}$ at each scenario $\xi\in\Xi$,
while \eqref{rest1_prob_tr} are the demand constraints.

We solve Problem~\ref{prob_transporte} following the structure of the 
problem in \eqref{e:pricp} with $T=\Id$. We consider the following two 
equivalent formulations.
\subsubsection{Primal-Dual Formulation}
Note that Problem~\ref{prob_transporte} can be equivalently written as
\begin{align*}
\label{prob_esto_aplic}\tag{$\mathcal{P}$}
\mini_{(x,f)\in {\mathbb{R}}^{|{\mathcal{A}}||\Xi|}\times 
{\mathbb{R}}^{|R||\Xi|}} 
F(x,f)+G(L(x,f))+H(x,f),
\end{align*}
where 
\begin{equation}
\label{e:defsform1}
\begin{cases}
(\forall \xi\in\Xi)\quad V^{+}_\xi:=\Big\{f\in {\mathbb{R}}_{+}^{|R|}\,:\, 
(\forall (o,d)\in {\mathcal{O}}\times{\mathcal{D}})\,\, \sum_{r\in 
R_{od}} 
f_{r} =h_{od,\xi}\Big\}\\[3mm]
\Lambda:=(D^{|\Xi|}\cap \mathcal{N})\times 
\Big(\bigtimes_{\xi\in\Xi} V_{\xi}^{+}\Big)\\[3mm]
F:=\iota_{\Lambda}\\[2mm]
(\forall \xi\in\Xi)\quad \Theta_{\xi}:=\left\{(x,u)\in 
{\mathbb{R}}^{|{\mathcal{A}}|}\times{\mathbb{R}}^{|{\mathcal{A}}|}\,:\, 
(\forall 
a\in {\mathcal{A}})\,\, u_{a} -x_{a}\leq c_{a,\xi}\right\}\\[3mm]
G:=\iota_{\bigtimes_{\xi \in \Xi} \Theta_{\xi}}\\[2mm]
L\colon (x,f)\mapsto (x_{\xi},Nf_{\xi})_{\xi\in \Xi}\\[2mm]
H\colon (x,f) \mapsto \displaystyle\sum_{\xi\in \Xi} 
p_{\xi}\left[\sum_{a\in {\mathcal{A}}} 
\displaystyle\int_{0}^{(Nf_{\xi})_a} 
t_{a,\xi}(z)dz+\frac{1}{2}x_{\xi}^{\top}Qx_{\xi}\right].
\end{cases}
\end{equation}
Observe that
$F$ and 
$G$ are lower semicontinuous convex proper functions, and  
$L$ is linear and bounded
with $\|L\|\leq \max\left\{1,\|N\|\right\}$. Moreover, 
note that since $(t_{a,\xi})_{a\in\mathcal{A},\xi\in\Xi}$ are increasing, 
$N$ is linear, and $Q$ is definite positive, $H$ is a separable convex 
function. In addition, by defining $$\psi\colon f\mapsto 
(p_{\xi}N^{\top}(t_{a,\xi}((Nf_{\xi})_a))_{a\in
{\mathcal{A}}})_{\xi \in \Xi},$$ simple computations yield
\begin{equation*}
\nabla 
H\colon (x,f)\mapsto \left(\left(p_{\xi}Qx_{\xi}\right)_{\xi\in 
\Xi},\psi(f)\right),
\end{equation*}
which  is Lipschitz continuous with constant
\begin{equation*}
\beta^{-1}:=\max_{\xi \in \Xi} \left( p_{\xi}
\max\Big\{\|Q\|,\|N\|^{2}\max_{a\in {\mathcal{A}}} 
\chi_{a,\xi}\Big\}\right).
\end{equation*}
Altogether,
\eqref{prob_esto_aplic} is a 
particular instance of problem in \eqref{e:pricp} with 
$V={\mathbb{R}}^{|{\mathcal{A}}||\Xi|}\times 
{\mathbb{R}}^{|R||\Xi|}$ and $\ell=\iota_{\{0\}}$.
Therefore, this formulation satisfies the hypotheses in 
\cite{10.condat} its algorithm reduces to	
\begin{equation}
\label{alg_ejem_direct}
(\forall k\in {\mathbb{N}})\quad
\left\lfloor			
\begin{array}{ll}
\text{For every }\xi\in\Xi\\
\left\lfloor
\begin{array}{ll}
(\widetilde{p}_{\xi}^{k},\widetilde{u}_{\xi}^{k})=(p_{\xi}^{k}+\gamma\overline{x}_{\xi}^{k},u_{\xi}^{k}+\gamma
N\overline{f}_{\xi}^{k})\\
(p_{\xi}^{k+1},u_{\xi}^{k+1})
=(\widetilde{p}_{\xi}^{k},\widetilde{u}_{\xi}^{k})-\gamma				
P_{\Theta_{\xi}}(\gamma^{-1}(\widetilde{p}_{\xi}^{k},\widetilde{u}_{\xi}^{k}))
\end{array}
\right.\\
\widetilde{x}^{k}=x^{k}-\tau (p^{k+1}+ 
(p_{\xi}Qx_{\xi}^{k})_{\xi\in \Xi})\\
\widetilde{f}^{k}=f^{k}-\tau (N^{\top}u^{k+1} + \psi(f^{k}))\\
x^{k+1}=P_{D^{|\Xi|}\cap \mathcal{N}}\widetilde{x}^{k}\\
(\forall \xi\in\Xi)\quad 
f_{\xi}^{k+1}=P_{V^{+}_{\xi}}\widetilde{f}_{\xi}^{k}\\
\overline{x}^{k+1}=2x^{k+1}-x^{k}\\
\overline{f}^{k+1}=2f^{k+1}-f^{k},
\end{array}
\right.
\end{equation}
where $(x^0,f^0)\in \RR^{|{\mathcal{A}}||\Xi|} \times \RR^{|R||\Xi|}$, 
$(\overline{x}^{0},\overline{f}^{0})=(x^{0},f^{0})$, 
$(p^0,u^0)\in {\mathbb{R}}^{|{\mathcal{A}}||\Xi|}\times 
{\mathbb{R}}^{|{\mathcal{A}}||\Xi|}$, and $\tau \in 
\left]0,2\beta\right[$ and
$\gamma\in\RPP$ are such that 
\begin{equation}
\label{e:condpartr}
\tau\gamma\max\left\{1,\|N\|^{2}\right\}<1-\tau/2\beta.
\end{equation}
\begin{remark} The projections 
$(P_{\Theta_{\xi}})_{\xi\in\Xi}$,  
$P_{D^{|\Xi|}\cap \mathcal{N}}$, and 
$(P_{V^{+}_{\xi}})_{\xi\in\Xi}$ appearing in 
\eqref{alg_ejem_direct}, can be computed efficiently, as we detail 
below. 
\begin{enumerate}
	\item[(i)] Let $\xi\in \Xi$ and let
	$(x,u)\in \mathbb{R}^{|{\mathcal{A}}|} \times 
	\mathbb{R}^{|{\mathcal{A}}|}$. We deduce from 
	\cite[Proposition~29.3 \& Example~29.20]{19.Livre1} that 
	$P_{\Theta_{\xi}}(x,u)=(P_{a,\xi}(x,u))_{a\in\mathcal{A}}$, 
	where  
	\begin{equation*}
	(\forall a \in 
	{\mathcal{A}})\,\, P_{a,\xi}(x,u)=
	\begin{cases}
	\left(\frac{x_{a}+u_{a}-c_{a,\xi}}{2},\frac{x_{a}+u_{a}
		+c_{a,\xi}}{2}\right),
	&\mbox{if } u_{a}-x_{a}-c_{a,\xi}>0;\\
	(x_{a},u_{a}), &\mbox{otherwise.} 
	\end{cases}
	\end{equation*}
	\item[(ii)] Let $\xi\in \Xi$ and note 
	that
	\begin{equation*}
	V^{+}_{\xi}=\displaystyle\bigtimes_{(o,d)\in {\mathcal{O}}\times 
		{\mathcal{D}}} V_{od,\xi},
	\end{equation*}
	where, for every $(o,d)\in\mathcal{O}\times\mathcal{D}$, 
	$V_{od,\xi}:=\{f\in {\mathbb{R}}_{+}^{|R_{od}|}\,:\, 
	\sum_{r\in R_{od}} 
	f_{r}=h_{od,\xi}\}$. It follows from
	\cite[Proposition~29.3]{19.Livre1} that 
	\begin{equation*}
	P_{V_{\xi}^{+}}\colon 
	f=(f_{od})_{\substack{o\in\mathcal{O}\\d\in\mathcal{D}}}\mapsto
	(P_{V_{od,\xi}}f_{od})_{\substack{o\in\mathcal{O}\\d\in\mathcal{D}}}.
	\end{equation*}
	For every $(o,d)\in\mathcal{O}\times\mathcal{D}$, the projection 
	$P_{V_{od,\xi}}$ 
	can be computed efficiently by using the 
	quasi-Newton algorithm developed in 
	\cite{calculo_proj_simplex}.
	\item[(iii)]
	Note that
	\begin{equation*}
	D^{|\Xi|}\cap \mathcal{N}=\displaystyle\bigtimes_{a\in 
		{\mathcal{A}}} C_a,
	\end{equation*}
	where, for every $a\in\mathcal{A}$, $C_a=\{y\in 
	\left[0,M_{a}\right]^{|\Xi|}\,:\, 
	(\forall (\xi,\xi')\in \Xi^{2})\:\: 
	y_{\xi}=y_{\xi'}\}$.
	It follows from \cite[Proposition~29.3]{19.Livre1} that 
	$
	P_{D^{|\Xi|}\cap \mathcal{N}}\colon 
	x=(x_a)_{a\in\mathcal{A}}\mapsto
	(P_{C_a}x_a)_{a\in\mathcal{A}}
	$
	and
	\begin{equation}
	\label{proj_C_a}
	(\forall a\in\mathcal{A})\quad P_{C_{a}}\colon y\mapsto 
	\midd(0,\overline{y},M_{a})\boldsymbol{1},
	\end{equation}
	where $\boldsymbol{1}=(1,\ldots,1)^{\top}\in\RR^{|\Xi|}$, 
	$\overline{y}=\frac{1}{|\Xi|}\sum_{\xi \in \Xi} y_{\xi}$, and 
	$\midd(a,b,c)$ is the middle value among $a$, 
	$b$, and $c$. In order to prove \eqref{proj_C_a}, let $a\in 
	{\mathcal{A}}$, 
	$y\in 
	\RR^{|\Xi|}$,
	and set 
	$\widehat{\theta}:=\midd(0,\overline{y},M_{a})\in\left[0,M_a\right]$. 
	For every $x\in C_a$, we have $x=\eta\boldsymbol{1}$ for some 
	$\eta\in [0,M_{a}]$ and
	\begin{align*}
	(y-\widehat{\theta}\boldsymbol{1})^{\top}
	(x-\widehat{\theta}\boldsymbol{1})
	&=(\eta-\widehat{\theta})|\Xi|(\overline{y}-\widehat{\theta})\nonumber\\
	&=
	\begin{cases}
	\eta|\Xi|\overline{y},\quad&\text{if}\:\:\overline{y}<0;\\
	0,&\text{if}\:\:\overline{y}\in[0,M_{a}];\\
	(\eta-M_{a})|\Xi|(\overline{y}-M_{a}),&\text{if}\:\:\overline{y}>M_{a},
	\end{cases}
	\end{align*}
	which yields $(y-\widehat{\theta}\boldsymbol{1})^{\top}
	(x-\widehat{\theta}\boldsymbol{1})\le 0$ and obtain from 
	\cite[Theorem~3.16]{19.Livre1} that
	$P_{C_a}y=\widehat{\theta}\boldsymbol{1}$.
\end{enumerate}
\end{remark}

\subsubsection{Vector Subspace Primal-Dual Formulation}
For the second equivalent formulation of  
Problem~\ref{prob_transporte} 
consider the  closed vector subspace
\begin{equation*}
S=\left\{f\in {\mathbb{R}}^{|R||\Xi|} \,:\, (\forall \xi\in \Xi)(\forall 
(o,d)\in {\mathcal{O}}\times{\mathcal{D}})\, \displaystyle\sum_{r\in 
R_{od}} 
f_{r,\xi} =0\right\}
\end{equation*}
and let ${\widehat{f}}^{0}$ defined by
\begin{equation*}
(\forall \xi \in \Xi)(\forall (o,d)\in{\mathcal{O}}\times 
{\mathcal{D}})(\forall r\in 
R_{od})\quad \widehat{f}^{0}_{r,\xi}=h_{od,\xi}/{|R_{od}|},
\end{equation*}
which
satisfies
\eqref{rest1_prob_tr}. Then, under the notation in \eqref{e:defsform1}, 
the Problem~\ref{prob_transporte} is equivalent to
\begin{align*}
\label{prob_esto_aplic_sev}\tag{$\mathcal{P}_{V}$}
\mini_{(x,f)\in \mathcal{N}\times S} 
\widehat{F}(x,f+\widehat{f}^0)+
{G}(L(x,f+\widehat{f}^0))+{H}(x,f+\widehat{f}^0),
\end{align*}
where $\widehat{F}=\iota_{D^{|\Xi|}\times 				
{\mathbb{R}}_{+}^{|R||\Xi|}}$. Note that, the difference with respect 
to \eqref{prob_esto_aplic} is that in \eqref{prob_esto_aplic_sev}
we propose a vector subspace splitting on function $F$ defined in 
\eqref{e:defsform1}.

In addition, observe that $\widehat{F}(\cdot+(0,\widehat{f}^0))$ and 
${G}(\cdot+L(0,\widehat{f}^0))$ are lower semicontinuous, convex, 
and proper, and $H(\cdot+(0,\widehat{f}^0))$ 
is convex differentiable with $\beta^{-1}-$Lipschitz gradient.
Thus, \eqref{prob_esto_aplic_sev} 
satisfies the hypotheses of problem \eqref{e:pricp} with 
$V=\mathcal{N}\times 
S$ and $\ell=\iota_{\{0\}}$. Hence, by using 
\cite[Proposition~29.1(i)]{19.Livre1}, the algorithm in \eqref{alg_opt} 
with $T=\Id$ reduces to
\begin{equation}
\label{alg_ejem_2}
(\forall k\in {\mathbb{N}})\quad
\left\lfloor
\begin{array}{ll}
\widetilde{p}^{k}=p^{k}+\gamma \overline{x}^{k}\\
\widetilde{u}^{k}=u^{k}+\gamma N(\overline{f}^{k}+\widehat{f}^{0})\\
(\forall \xi\in \Xi)\quad
(p_{\xi}^{k+1},u_{\xi}^{k+1})=(\widetilde{p}_{\xi}^{k},
\widetilde{u}_{\xi}^{k})-\gamma 
P_{\Theta_{\xi}}(\gamma^{-1}(\widetilde{p}_{\xi}^{k},\widetilde{u}_{\xi}^{k}))\\
\widetilde{x}^{k}=x^{k}+\tau y^{k}-\tau P_{\mathcal{N}}( 
p^{k+1}+ 
(p_{\xi}Qx_{\xi}^{k})_{\xi\in\Xi})\\
\widetilde{f}^{k}=f^{k}+\tau g^{k}-\tau P_{S}(N^{\top}u^{k+1} + 
\psi(f^{k}+{\widehat{f}}^{0}))\\
z^{k+1}=P_{D^{|\Xi|}}\widetilde{x}^{k}\\
\ell^{k+1}=P_{\RR_{+}^{|R||\Xi|}} ({\widetilde{f}}^{k}+{\widehat{f}}^{0}) - 
{\widehat{f}}^{0}\\
x^{k+1}=P_{\mathcal{N}}z^{k+1}\\
f^{k+1}=P_{S}\ell^{k+1}\\
y^{k+1}=y^{k}+(x^{k+1}-z^{k+1})/{\tau}\\
g^{k+1}=g^{k}+(f^{k+1}-\ell^{k+1})/{\tau}\\
\overline{x}^{k+1}=2x^{k+1}-x^{k}\\
\overline{f}^{k+1}=2f^{k+1}-f^{k},
\end{array}
\right.
\end{equation}
where $(x^0,f^0)=(\overline{x}^{0},\overline{f}^{0})\in 
\mathcal{N}\times S$, $(p^0,u^0)\in 
{\mathbb{R}}^{|{\mathcal{A}}||\Xi|}\times 
{\mathbb{R}}^{|{\mathcal{A}}||\Xi|}$,
 $(y^{0},g^{0})\in\mathcal{N}^{\perp}\times S^{\perp}$, and $\tau \in 
		\left]0,2\beta\right[$ and  
		$\gamma \in\RPP$ satisfy \eqref{e:condpartr}.

\begin{remark}
The projections appearing in \eqref{alg_ejem_2} are explicit. Indeed,
for every $x\in {\mathbb{R}}^{|{\mathcal{A}}||\Xi|}$ and
$\xi\in\Xi$, we have that
$(P_{\mathcal{N}}x)_{\xi}=\frac{1}{|\Xi|}\sum_{\xi'\in \Xi} 
x_{\xi'}$ and $(P_{D^{|\Xi|}}x)_{\xi} = 
(\midd(0,x_{a,\xi},M_{a}))_{a\in{\mathcal{A}}}$. Moreover, for every 
$f\in 
{\mathbb{R}}^{|R||\Xi|}$ we have, for every 
$(o,d)\in\mathcal{O}\times\mathcal{D}$ and $r\in R_{od}$,  
$(P_{S}f)_r=(f_{r,\xi} - 
\frac{1}{|R_{od}|}\sum_{r'\in R_{od}} 
f_{r',\xi})_{\xi\in\Xi}$ and 
$(P_{\RR_{+}^{|R||\Xi|}}f)_r=(\max\{0,f_{r,\xi}\})_{ \xi\in\Xi}$.
\end{remark}

\subsubsection{Numerical Experiences}
In this subsection we compare the efficiency of the algorithms in
\eqref{alg_ejem_direct} and \eqref{alg_ejem_2} to solve the arc 
capacity expansion problem. We consider two networks used in 
\cite{aplicacion_julio}. Network 1, represented 
in Figure~\ref{fig:grafo1}, has $7$ arcs 
and $6$ paths and Network 2, in Figure~\ref{fig:grafo2}, has
$19$ arcs and $25$ paths.
\begin{figure}[H]
\begin{center}
	\includegraphics[scale=0.6]{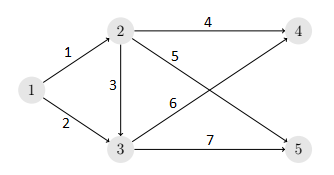}
\end{center}
\caption{Network 1 \cite{aplicacion_julio}}
\label{fig:grafo1}
\end{figure}
\vspace{-0.5cm}
\begin{figure}[H]
\centering
\includegraphics[scale=0.6]{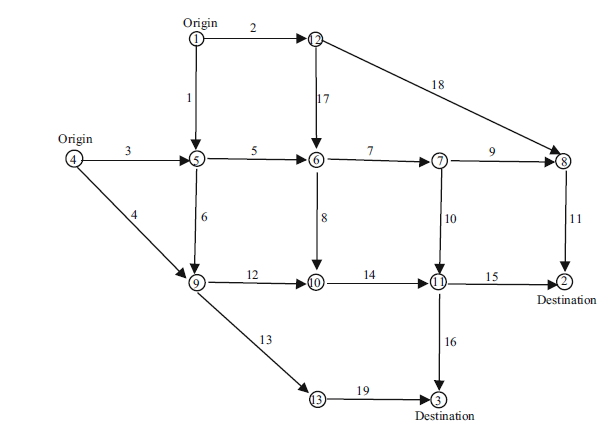}
\caption{Network 2 \cite{aplicacion_julio}}
\label{fig:grafo2}
\end{figure}
In our numerical experiences we set $p_{\xi}\equiv 1/{|\Xi|}$,
$(c_{\xi})_{\xi\in\Xi}$ as a sample of the random variable $100\cdot 
b+d\cdot \mbox{beta}(2,2)$, where
\begin{equation*}
b=
\left\{\begin{array}{ll}
(1,1,2,2,1,1,1) \quad\text{in Network~1}\\
(10, 4.4, 1.4, 10, 3, 
4.4, 10, 2, 2, 4, 7, 7,7, 7, 4, 3.5, 2.2, 4.4, 7)\,\,\text{in Network~2}
\end{array}
\right.
\end{equation*}
and
\begin{equation*}
d=\begin{cases}
(15,15,30,30,15,15,15) \quad\text{in Network~1}\\
(15, 6.6, 2.1, 15, 4.5, 6.6, 15, 3, 3, 6, 10.5,\\
\hspace{2cm}10.5, 10.5, 10.5, 6, 5.25, 3.3, 6.6, 10.5)
\quad\text{in Network~2},
\end{cases}
\end{equation*}
and $(h_{\xi})_{\xi\in\Xi}$ as a sample of the random variable
\begin{equation*}
h=
\begin{cases}
(h_{1,4},h_{1,5})\sim(150,180) + 
(120,96)\cdot\mbox{beta}(5,1) \quad\text{in Network~1}\\
(h_{1,2},h_{1,3},h_{4,2},h_{4,3})
\sim(300,700,500,350)\\
\qquad\qquad\qquad\qquad+ (120,120,120,120)\cdot 
\mbox{beta}(50,10) 
\quad\text{in Network~2}.
\end{cases}
\end{equation*}
We set the capacity limits  $(M_a)_{a\in\mathcal{A}}=\theta d$,
where $\theta=40$ in Network~1 and $\theta=200$ 
in Network~2. We also set $Q=\Id\in 
\RR^{|{\mathcal{A}}|\times|{\mathcal{A}}|}$ and, for every
$ a\in {\mathcal{A}}$ and $\xi \in \Xi$, the travel time function is $ 
t_{a,\xi}\colon u \mapsto 
\eta_{a}(1+0.15\,u/c_{a,\xi})$, where 
\begin{equation*}
\eta=
\begin{cases}
(6,4,3,5,6,4,1) \quad\text{in Network~1}\\
(7, 9, 9, 12, 3, 9, 5, 13, 5, 9, 9, 10, 9, 6, 9, 8, 7, 14, 11)\quad\text{in 
Network~2}.
\end{cases}
\end{equation*}
We implement the algorithms in \eqref{alg_ejem_direct} and 
\eqref{alg_ejem_2} for different values of $|\Xi|$. 
We obtain the following results by considering
$ 20 $ random realizations of $(c_{\xi})_{\xi\in\Xi}$ and 
$(h_{\xi})_{\xi\in\Xi}$.
\begin{table}[H]
\caption{Average execution time (number of iterations) with 
	relative 
	error tolerance $e=10^{-10}$}
\tabcolsep=0.11cm
\begin{tabular}{ccccc}
	\hline
	\multicolumn{1}{|c|}{Network 1} & \multicolumn{1}{c|}{$|\Xi|=1$} & 
	\multicolumn{1}{c|}{$|\Xi|=3$} & \multicolumn{1}{c|}{$|\Xi|=5$} & 
	\multicolumn{1}{c|}{$|\Xi|=10$} \\ \hline
	\multicolumn{1}{|c|}{algorithm \eqref{alg_ejem_direct}} & 
	\multicolumn{1}{c|}{0.082 (1143)} & \multicolumn{1}{c|}{0.731 
	(3217)} 
	& \multicolumn{1}{c|}{1.363 (4199)} & \multicolumn{1}{c|}{4.388 
		(5698)} \\ \hline
	\multicolumn{1}{|c|}{algorithm \eqref{alg_ejem_2}} & 
	\multicolumn{1}{c|}{0.075 (1160)} & \multicolumn{1}{c|}{0.607 
	(3284)} 
	& \multicolumn{1}{c|}{1.098 (4294)} & \multicolumn{1}{c|}{3.485 
		(5804)} \\ \hline
	\multicolumn{1}{|c|}{\% improvement of time} & 
	\multicolumn{1}{c|}{8.54\%} & \multicolumn{1}{c|}{16.96\%} & 
	\multicolumn{1}{c|}{19.44\%} & \multicolumn{1}{c|}{20.58\%} \\ \hline
	\multicolumn{1}{l}{} & \multicolumn{1}{l}{} & \multicolumn{1}{l}{} & 
	\multicolumn{1}{l}{} & \multicolumn{1}{l}{} \\ \hline
	\multicolumn{1}{|c|}{Network 2} & \multicolumn{1}{c|}{$|\Xi|=1$} & 
	\multicolumn{1}{c|}{$|\Xi|=3$} & \multicolumn{1}{c|}{$|\Xi|=5$} & 
	\multicolumn{1}{c|}{$|\Xi|=10$} \\ \hline
	\multicolumn{1}{|c|}{algorithm \eqref{alg_ejem_direct}} & 
	\multicolumn{1}{c|}{0.864 (4801)} & \multicolumn{1}{c|}{10.195 
		(27285)} & \multicolumn{1}{c|}{16.166 (27660)} & 
	\multicolumn{1}{c|}{45.327 (39790)} \\ \hline
	\multicolumn{1}{|c|}{algorithm \eqref{alg_ejem_2}} & 
	\multicolumn{1}{c|}{0.637 (4816)} & \multicolumn{1}{c|}{7.627 
	(28147)} 
	& \multicolumn{1}{c|}{12.069 (28885)} & \multicolumn{1}{c|}{33.204 
		(40848)} \\ \hline
	\multicolumn{1}{|c|}{\% improvement of time} & 
	\multicolumn{1}{c|}{26.27\%} & \multicolumn{1}{c|}{25.19\%} & 
	\multicolumn{1}{c|}{25.34\%} & \multicolumn{1}{c|}{26.75\%} \\ \hline
\end{tabular}
\end{table}
\raggedbottom
Note that the algorithm with vector subspaces in \eqref{alg_ejem_2} is 
more time efficient 
compared to the classical primal-dual algorithm in 
\eqref{alg_ejem_direct}. Indeed, the percentage of improvement 
reaches up to $26.75\%$, for the larger dimensional case of  
Network~2 and $\Xi$. It is worth to notice that the number of iterations 
is lower 
in average for the approach without vector subspaces, but it is 
explained by the subroutines that compute the projections onto 
$(V_{\xi}^+)_{\xi\in\Xi}$,
which lead to a larger computational time by iteration.
In order to show
the difference of both algorithms, in Figure~\ref{fig:ntw1} and 
Figure~\ref{fig:ntw2} we illustrate 
the relative 
error  depending on the execution time for the best and the worst 
instance respect to convergence time.
\vspace{-0.3cm}
\begin{figure}[H]
\centering
\begin{subfigure}{.5\textwidth}
  \centering
  \includegraphics[width=1\linewidth]{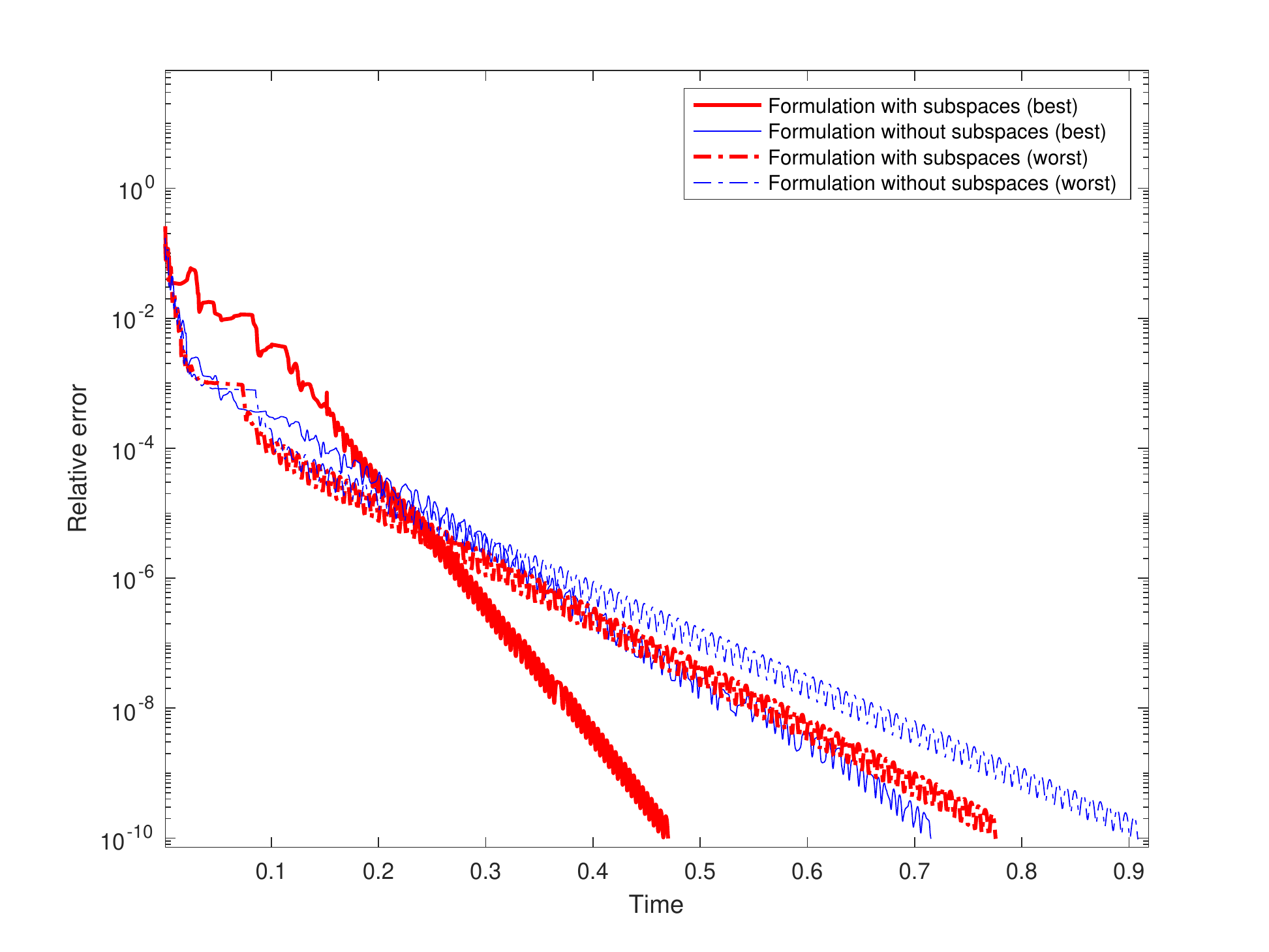}
\end{subfigure}%
\begin{subfigure}{.5\textwidth}
  \centering
  \includegraphics[width=1\linewidth]{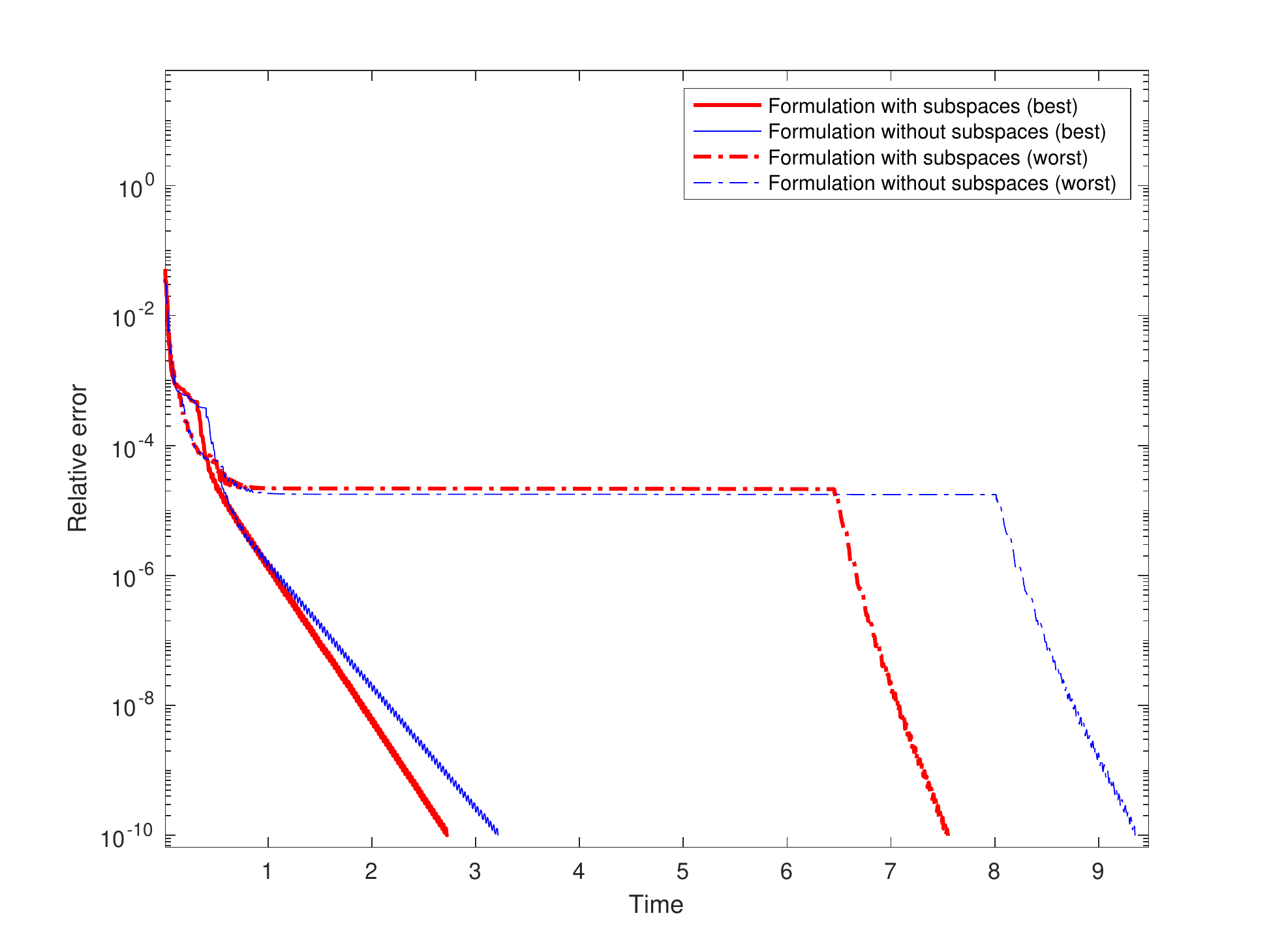}
\end{subfigure}%
\caption{Relative error in Network 1 (semilog) with $|\Xi|=3$ 
	(left) and $|\Xi|=10$ (right)}
\label{fig:ntw1}
\end{figure}
\vspace{-0.5cm}
\begin{figure}[H]
\centering
\begin{subfigure}{.5\textwidth}
  \centering
  \includegraphics[width=1\linewidth]{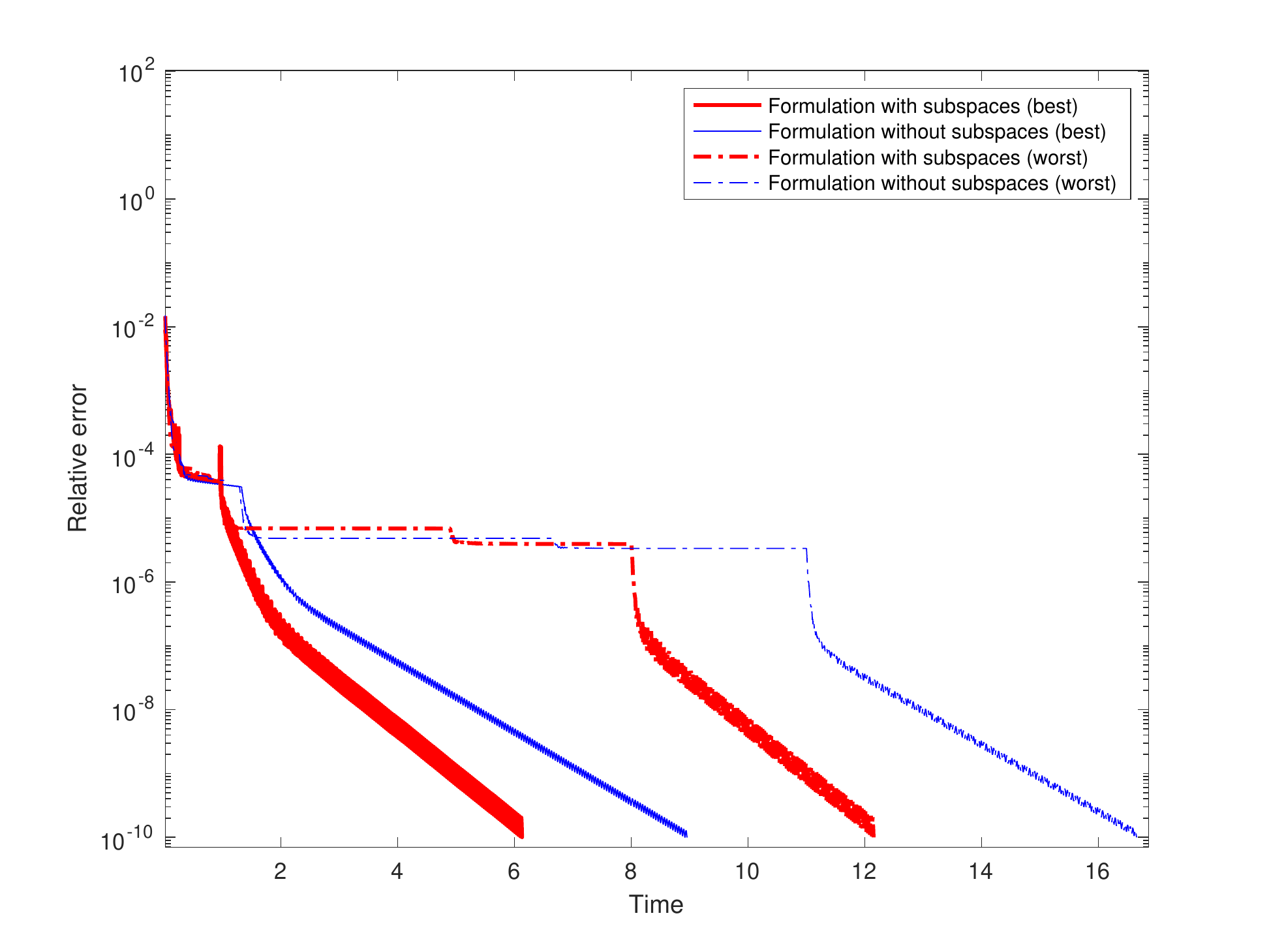}
\end{subfigure}%
\begin{subfigure}{.5\textwidth}
  \centering
  \includegraphics[width=1\linewidth]{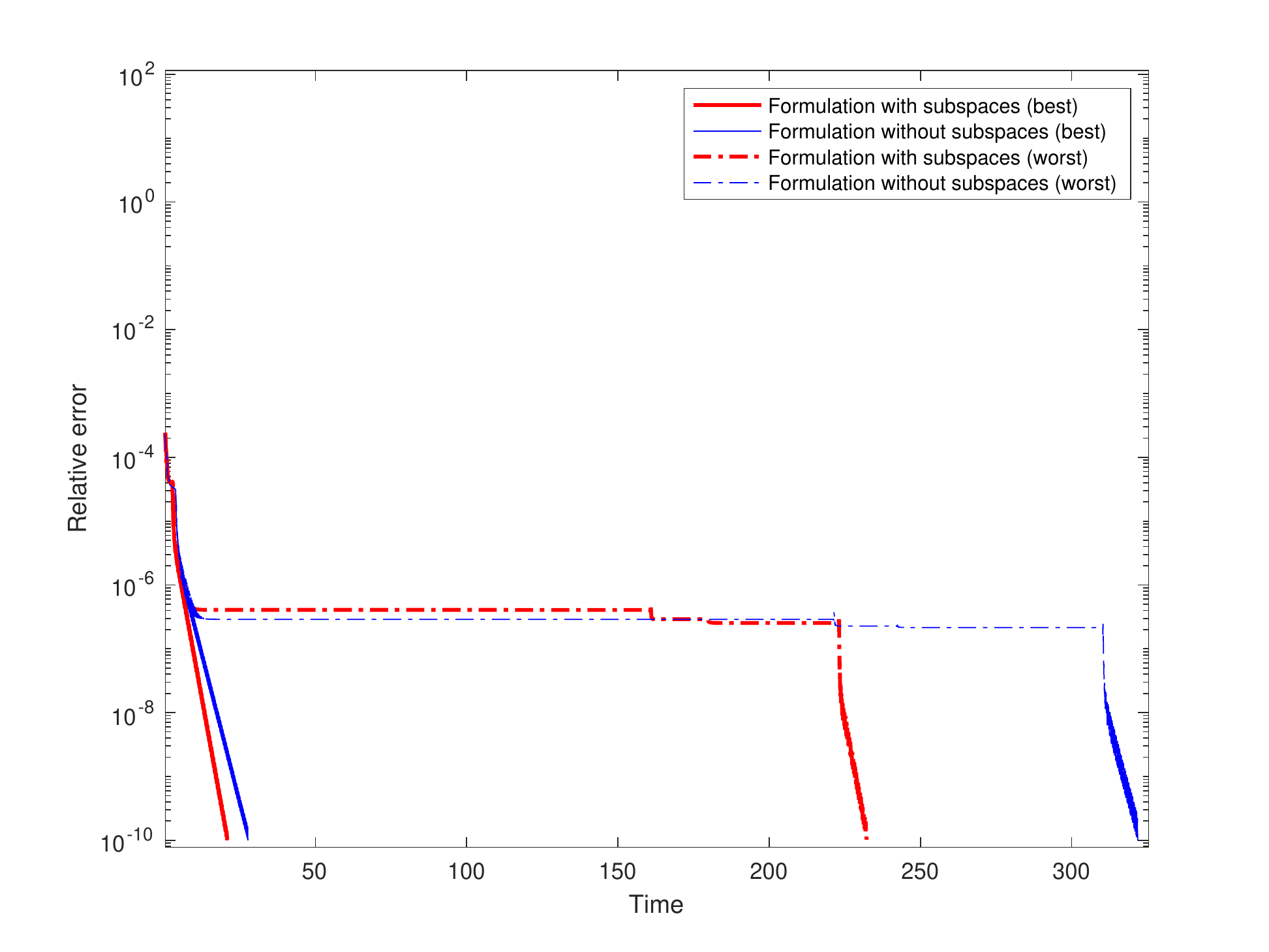}
\end{subfigure}%
\caption{Relative error in Network 2 (semilog) with $|\Xi|=3$ 
	(left) and $|\Xi|=10$ (right)}
\label{fig:ntw2}
\end{figure}
\subsection{Stationary MFG with non-local 
couplings}
\label{sec:numer_mfg}
Let us consider the following second order ergodic MFG
system \cite{Lasry_Lions_2006i,Lasry_Lions_2007}
\begin{equation}\label{MFG_ergodic_system}\begin{array}{rll}
- \nu \Delta u + \mathscr{H}(x,\nabla u) + \lambda \; 
&= \phi(x,m)  &\text{in } \TT^d, \\[5pt]
- \nu \Delta m - \text{div}\left( \partial_{p} \mathscr{H}(x, \nabla 
u)m\right) \; \; &=0 \qquad &\text{in } \TT^d, \\[5pt]
\int_{\TT^d} u(x) {\rm d} x=0, \:\: m \geq 0,  
&\:\int_{\TT^d}m(x) 
{\rm d} x&=1. 
\end{array}
\end{equation}
In the system above, $\TT^d$ denotes the $d$-dimensional torus and  
the unknowns are $u: \TT^d \to \RR$, $m: \TT^d \to \RR$, and 
$\lambda \in \RR$. The function $\mathscr{H}: \TT^d \times \RR^d \to 
\RR$ is the so-called Hamiltonian function,   $\phi: \TT^d \times 
L^1(\TT^d) \to \RR$ and  $\nu \in \RPP$. We assume that, for all $x\in 
\TT^d$, the function $p \mapsto \mathscr{H}(x,p)$ is 
convex. Existence and uniqueness results for solutions $(u,m)$ to 
system \eqref{MFG_ergodic_system} have been shown, under 
suitable assumptions,  in 
\cite{Lasry_Lions_2006i,Lasry_Lions_2007,Bardi_Feleqi_2016,meszaros_silva_2018}.
 The function $\phi$ is called coupling or interaction term and system 
\eqref{MFG_ergodic_system} is called a MFG system with {\it 
non-local coupling}. In contrast, the {\it local coupling} case 
corresponds to \eqref{MFG_ergodic_system} when  
$\phi: \TT^d \times \RP \to \RR$ and the right-hand-side of the first 
equation is replaced by 
$\phi(x,m(x))$. 

For the sake of simplicity, we will assume that  
\begin{equation}
\label{e:defk}
\mathscr{H}\colon(x,p) \mapsto \frac{|p|^2}{2} \quad 
\text{and} \quad \phi\colon (x,m) \mapsto \int_{\TT^d} k(x,y) m(y) {\rm 
d}y+k_0(x),
\end{equation}
where $k_0: \TT^d \to \RR$ is a Lipschitz function  and $k: 
\TT^d 
\times \TT^d \to \RR$ is a smooth function satisfying that, for every 
finite selection of 
points $(x_r)_{1\le r\le \ell}$ in $\TT^d$ and for every $m\colon 
\TT^d\to\RR$, we have
\begin{equation}
\label{e:PDS}
\sum_{r=1}^{\ell}\sum_{s=1}^{\ell}k(x_r,x_s)m(x_r)m(x_s)\ge 0
\end{equation}
and $k(x_r,x_s)=k(x_s,x_r)$ for every $r$ and $s$ in $\{1,\ldots,\ell\}$.
The condition \eqref{e:PDS} is known as the Positive Definite 
Symmetric (PDS)
property \cite[Section~6.2]{Mohri} and implies  that $k$ satisfies the 
monotonicity 
condition
\begin{equation}
\int_{\TT^d \times \TT^d} k(x,y)(m_1(x) - m_2(x))(m_1(y) - m_2(y))  
{\rm d} x {\rm d} y \geq 0,
\end{equation}
for all $m_1$ and $m_2$ in $L^1(\TT^d)$ such that 
$\int_{\TT^d} 
m_1 (x) {\rm d} x=\int_{\TT^d} m_2 (x) {\rm d} x=1$.
Under the previous assumptions, the results in 
\cite{Lasry_Lions_2006i,Lasry_Lions_2007} ensure the existence of 
an 
unique classical solution $(u,m)$ to \eqref{MFG_ergodic_system}.
Moreover, it follows from \eqref{e:defk} and \eqref{e:PDS} that
\eqref{MFG_ergodic_system} admits a variational formulation, i.e., it 
corresponds to the optimality condition of an optimization problem 
\cite{Lasry_Lions_2007}. 

The numerical approximation of solutions to variational MFGs with 
local 
couplings has been 
addressed in 
\cite{aime10,benamoucarlier15,Benamou_et_al_2017,Andreev16,BAKS}.
We focus in the case when $\phi(x,\cdot)$ is non-local
in which previous algorithms are not applicable.
In this context, 
\cite{Nurbekyan_et_all_2020} solves \eqref{MFG_ergodic_system} via 
approximated 
solutions to 
the analogous time-dependent MFG system as the time horizon goes 
to infinity. In contrast, we propose algorithms following a direct 
approach based on the 
variational formulation 
of \eqref{MFG_ergodic_system}, in which the vector 
subspace structure arises.

\subsubsection{Finite difference approximation}
We now introduce  
a discretization of system  \eqref{MFG_ergodic_system} in the 
two-dimensional case, following 
\cite{Achdou_Capuzzo_Dolcetta_2010}.  Let $N \in \NN$, set $h= 
1/N$, set $\mathcal{I}_N=\{0,\ldots,N-1\}$, and consider 
the uniform grid $\TT^2_{h}= \{x_{i,j}= (hi,hj) \; | \; (i, j)\in 
\mathcal{I}_N^2\}$ on $\TT^2$. Let us define $\M_{h}$ as the set of 
real valued functions defined on  $\TT^2_{h}$, $\mathcal{W}_h= 
\M_h^4$,
$\boldsymbol{1}\in\mathcal{M}_h$ as the constant function 
equals 
to $1$, and set $\Y_h= \{z \in \M_h \; | \;  
\sum_{i,j}z_{i,j}=0\}$, where we denote, for every 
$(i,j)\in\mathcal{I}_N^2$, $z_{i,j}=z(x_{i,j})$. 
We define $D_1\colon \mathcal{M}_h\to\mathcal{Y}_h$,
 $D_2\colon \mathcal{M}_h\to\mathcal{Y}_h$, $D_h\colon 
 \mathcal{M}_h\to\mathcal{Y}_h^4$, $\Delta_h\colon 
 \mathcal{M}_h\to\mathcal{Y}_h$, and 
 $\mbox{div}_h\colon\mathcal{W}_h\to 
 \mathcal{Y}_h$ by
\begin{equation}
\label{e:defdiscr}
\begin{array}{l}
(D_1 z)_{i,j}= \frac{z_{i+1,j}-z_{i,j}}{h}, \quad (D_2 
z)_{i,j}=\frac{z_{i,j+1}-z_{i,j}}{h},\\[4pt]
[D_hz]_{i,j}=((D_1 z)_{i,j}, (D_1 z)_{i-1,j}, (D_2 z)_{i, j}, (D_2 
z)_{i,j-1}),\\[4pt]
(\Delta_h z)_{i,j}= \frac{z_{i-1,j} 
+z_{i+1,j}+z_{i,j-1}+z_{i,j+1}-4z_{i,j}}{h^2},\\[4pt]
(\text{div}_{h}(w))_{i,j}= (D_1 w^1)_{i-1, j} + 
(D_1w^2)_{i,j}+(D_2w^3)_{i,j-1}+ (D_2w^4)_{i,j},
\end{array}
\end{equation}
where previous definitions hold for every $i$ and $j$ in 
$\mathcal{I}_N$, $z\in\mathcal{M}_h$, $w\in\mathcal{W}_h$, and the 
sums between the 
indexes are taken modulo $N$. 
We also define $K_h\colon\M_h\to\M_h$ and $K_0\in\M_h$ by
\begin{equation}
\label{e:Ks}
(\forall m\in\M_h)(\forall (i,j)\in\mathcal{I}_N^2)\quad 
\begin{cases}
(K_hm)_{i,j}=h^2\sum_{i',j'} k(x_{i,j},x_{i',j'})m_{i',j'},\\
(K_0)_{i,j}=k_0(x_{i,j}).
\end{cases}
\end{equation}
Observe that condition in \eqref{e:PDS} implies that 
$K_h$ is positive semidefinite.
Let us consider the cone $C=\RP\times\RM\times\RP\times\RM$. The 
orthogonal projection onto $C^{N^2}$ is given by
\begin{equation}
\label{e:defC}\begin{array}{rcl}
(\forall w\in \mathcal{W}_h)(\forall (i,j)\in\mathcal{I}_N^2) \quad 
(P_{C^{N^2}}(w))_{i,j}&=&\big([w_{i,j}^1]_{+},[w_{i,j}^2]_{-},[w_{i,j}^3]_{+},
 
[w_{i,j}^4]_{-}\big),
\end{array}\end{equation}
where, for every $\xi\in\RR$, $[\xi]_+=\max\{0,\xi\}$ and
$[\xi]_-=\min\{0,\xi\}$.
The finite difference scheme  proposed in 
\cite{Achdou_Capuzzo_Dolcetta_2010} to approximate 
\eqref{MFG_ergodic_system} is the following: for every $i$ and $j$ in 
$\mathcal{I}_N$,
\begin{equation}\label{finite_difference_scheme} 
\begin{array}{rll}
-\nu(\Delta_h u)_{i,j} 
+\frac{1}{2}\big|P_{C}\big(-[D_hu]_{i,j}\big)\big|^2+\lambda&= 
(K_hm)_{i,j} +(K_0)_{i,j},  \\[6pt]
-\nu(\Delta_h m)_{i,j} +\big(\text{div}_h\big(m 
P_{C^{N^2}}\big(-[D_hu 
]\big)\big)\big)_{i,j}&=0, \\[6pt]
m_{i,j} \geq 0, \quad \displaystyle h^2 \sum_{i,j}m_{i,j}=1, \quad 
\sum_{i,j}u_{i,j}&=0, 
\end{array}
\end{equation}
where $(m,u)\in \mathcal{M}_h^2$ and $\lambda\in \RR$
are unknowns.

In order to obtain a variational interpretation of 
\eqref{finite_difference_scheme}, note that 
$(\M_h, \scal{\cdot}{\cdot})$, $(\mathcal{W}_h, 
\scal{\cdot}{\cdot}_{\mathcal{W}_h})$, and $(\Y_h, 
\scal{\cdot}{\cdot})$ are Hilbert spaces, where
$$\begin{array}{rll}
\scal{\cdot}{\cdot}\colon (m_1,m_2)&\mapsto  
\sum_{i,j=0}^{N-1} (m_1)_{i,j}(m_2)_{i,j},\\[6pt] 
\; \; \scal{\cdot}{\cdot}_{\mathcal{W}_h}\colon (w_1,w_2)&\mapsto 
\sum_{\ell=1}^{4} 
\scal{w_1^{\ell}}{w_2^{\ell}}.
\end{array}
$$
The adjoint operators $(-\Delta_h)^*: \Y_h \to \M_{h}$ 
and $(\mbox{div}_{h})^*: \Y_h \to  \mathcal{W}_h$  are given by
\begin{equation}
\label{e:adj}
(( i,j)\in \mathcal{I}_N^2) \; \;  
\left((-\Delta_h)^*u\right)_{i,j}= -(\Delta_{h} u)_{i,j}, \; \; 
((\text{div})^*u)_{i,j}= - [D_hu]_{i,j}. 
\end{equation}

\begin{remark}\label{injectivity_of_the_operators}
From the definition of $[D_h u]$ and the identity 
$$
\sum_{i,j=0}^{N-1} u_{i,j} (-\Delta_{h} u)_{i,j}= 
\sum_{i,j=0}^{N-1}\left[(D_{1}u)_{i,j}^2  + (D_{2}u)_{i,j}^2 \right],
$$
we have that both, $(\mbox{div}_{h})^*$ and 
$(-\Delta_h)^*$ are injective operators. Thus, we also 
have that both, $\mbox{div}_{h}$ and $-\Delta_h${\color{blue},}
are  surjective operators. 
\end{remark}

\subsubsection{Variational formulation, existence, and 
uniqueness}

Consider the function $b\colon \RR\times\RR^4\to\RX$ defined by 
\begin{equation}
\label{e:defb}
b\colon (\eta,\omega)\mapsto 
\begin{cases}
\frac{|\omega|^2}{2\eta},\quad &\text{ if }\:\: \eta>0\:\:\text{ and 
}\:\:\omega\in C;\\[2mm]
0,\quad &\text{ if }\:\: (\eta,\omega)=(0,0);\\
\pinf,&\text{ otherwise.}
\end{cases}
\end{equation}
Define the functions $B_h\colon \M_h \times \mathcal{W}_h \to \RX$ 
and $\Phi_h \colon \M_h \to \RR$ by
$$
\begin{array}{l}
B_h\colon (m, w) \mapsto \sum_{i,j=0}^{N-1}b(m_{i,j}, w_{i,j})    \\[6pt]
\Phi_{h} \colon m\mapsto  \frac{1}{2}\scal{m}{K_hm}+
\scal{K_0}{m},\\[6pt]
\end{array}
$$
where $K_0$ and $K_h$ are defined in \eqref{e:Ks}.
Note that, since $k$ is a PDS kernel, the function $\Phi_h$ is convex. 
We consider now the optimization problem
\begin{align}\label{discrete_optimization_problem}
\min_{(m,w)\in \M_h\times \mathcal{W}_h} \; B_h(m,w)+\Phi_h(m)\\
- \nu \Delta_h m + \mbox{div}_{h} w =0, \nonumber\\
h^2 \sum_{i,j} m_{i,j}=1.  \nonumber
\end{align}

We first provide existence and uniqueness of the solution to 
\eqref{discrete_optimization_problem} and to 
\eqref{finite_difference_scheme}, by assuming $\nu>0$ and without 
any strict convexity assumption. 
This result is interesting in its own right and the proof of the 
uniqueness does not follow the standard Lasry-Lions monotonicity 
argument (see e.g. \cite[Proposition 
3]{Achdou_Capuzzo_Dolcetta_2010}).
\begin{proposition} 
\label{p:uniq}
Let $\nu>0$. Then, there exists a unique solution 
$(\widehat{m},\widehat{u}, \widehat{\lambda})$ to system 
\eqref{finite_difference_scheme}. Moreover, $\widehat{m}$ is strictly 
positive, and
$(\widehat{m}, 
\widehat{m} P_{C}(-[D\widehat{u}]))$ is the unique solution to 
 \eqref{discrete_optimization_problem}.
\end{proposition}
\begin{proof} Denote by $\mathcal{S}$ the set of solutions to 
\eqref{discrete_optimization_problem}. The 
proofs that $\mathcal{S}\neq \emp$ and that if $(\widehat{m}, 
\widehat{w})\in \mathcal{S}$ then $\widehat{m}$ is strictly positive,   
follow 
exactly the same arguments than those in \cite[Theorem 2.1]{BAKS} 
and in \cite[Corollary 2.1]{BAKS}, respectively. Fix $(\widehat{m}, 
\widehat{w})\in \mathcal{S}$ and denote by $\Lambda(\widehat{m}, 
\widehat{w})$ the set of Lagrange multipliers at $(\widehat{m}, 
\widehat{w})$, i.e., the set of $(u, \lambda) \in \Y_h \times \RR$ such 
that 
\begin{equation}\label{optimality_conditions_proof}
\begin{array}{rcl}
\nu (- \Delta_h)^* u  +\lambda \mathbf{1}  - 
\nabla_{m}B_h(\widehat{m}, \widehat{w})&=& \nabla 
\Phi_{h}(\widehat{m}),\\[6pt]
(\mbox{div}_{h})^* u  &= & \nabla_{w} B_h(\widehat{m}, 
\widehat{w}),
\end{array}
\end{equation}
where we observe that $B_h(\widehat{m},\cdot)$ is differentiable 
since $\widehat{m}>0$.

The constraints of problem \eqref{discrete_optimization_problem}  
being affine, it follows from \cite[Lemma~2.2]{BAKS} and 
\cite[Fact~15.25(i)]{19.Livre1} that $\Lambda(\widehat{m}, 
\widehat{w})\ne\varnothing$. 
Moreover,  we deduce from Remark 
\ref{injectivity_of_the_operators} 
that, for every $u\in \Y_{h}$ and $\lambda\in\RR$ such that 
$(-\Delta_h)^*(u)+\lambda 
\mathbf{1}=0$, we have
\begin{equation}
0=\scal{u}{(-\Delta_h)^*(u)+\lambda 
\mathbf{1}}=\sum_{i,j=0}^{N-1}\left[(D_{1}u)_{i,j}^2
+ (D_{2}u)_{i,j}^2 \right],
\end{equation}
and hence $u=0$ and $\lambda=0$.
Therefore, 
$\Y_{h} \times \RR \ni (u,\lambda) \mapsto \nu 
(-\Delta_h)^*(u) +\lambda \mathbf{1} \in \M_h$ is injective and, 
from
the first equation of \eqref{optimality_conditions_proof}, we obtain
that $\Lambda(\widehat{m}, \widehat{w})$ is a singleton, say
$\Lambda(\widehat{m}, \widehat{w})=\{(\widehat{u}, 
\widehat{\lambda})\}$. Hence, the convexity of 
the problem in \eqref{discrete_optimization_problem}, 
\cite[Theorem~19.1]{19.Livre1}, and \cite[Lemma~2.2]{BAKS} imply 
that
\begin{equation}
\label{e:uniquelag}
(\forall (m',w')\in \mathcal{S})\quad 
\Lambda(m', w')=\Lambda(\widehat{m}, 
\widehat{w})=\{(\widehat{u}, 
\widehat{\lambda})\}.
\end{equation}
In addition, arguing as in the proof of 
\cite[Theorem 2.1{\rm(ii)}]{BAKS}, since $(\widehat{u}, 
\widehat{\lambda})$ is a solution to 
\eqref{optimality_conditions_proof}, we obtain that $(\widehat{m}, 
\widehat{u}, \widehat{\lambda})$ solves 
\eqref{finite_difference_scheme}. Conversely, if $(m',u',\lambda')$ 
solves \eqref{finite_difference_scheme},  by setting
$w'= m'P_{C}(-[D_hu'])$, we deduce that $(m',w',u',\lambda')$
solves \eqref{optimality_conditions_proof} and $(m',w')$ 
satisfies the constraints in 
\eqref{discrete_optimization_problem}. Since the latter is a convex 
optimization problem, we obtain from \cite[Theorem~19.1]{19.Livre1}  
and \eqref{e:uniquelag} 
that $(m', w') \in 
\mathcal{S}$ and
$(u',\lambda')\in \Lambda(m',w')=\{(\widehat{u}, 
\widehat{\lambda})\}$. Altogether, we have proved the
uniqueness of $(\widehat{u},\widehat{\lambda})$ in 
\eqref{finite_difference_scheme}. It only remains to prove 
the uniqueness of $\widehat{m}$, since it implies the 
uniqueness of $\widehat{w}$. 
For this purpose, 
define $E: \M_{h} \to \M_{h}$ as 
\begin{equation}
\label{e:defE}
(\forall (i,j)\in\mathcal{I}_N^2) \quad (Em)_{i,j}= -\nu(\Delta_h m)_{i,j} 
+\text{div}_h\big(m P_{C}\big(-[D_h\widehat{u} 
]\big)\big)_{i,j}.
\end{equation}
Note that the second equation in \eqref{finite_difference_scheme}
with $u=\widehat{u}$
is equivalent to $m\in\ker E$.
Then, since \cite[Fact~2.25]{19.Livre1} yields
$\M_{h}=\ran(E^*)\oplus\ker(E)$,
the rank–nullity theorem implies
$$\dim(\ran(E^*))+\dim(\ker(E))= \dim 
(\ker(E^{*})) + \dim(\ran(E^*))$$
and, hence, 
\begin{equation}\label{dim_1_TNI}\dim (\ker(E))=\dim  \ker(E^*).
\end{equation}
We claim that 
\begin{equation}
\label{e:claim}
\ker(E^*)=\menge{\alpha 
\mathbf{1}}{\alpha\in \RR},
\end{equation}
which implies that
there exists $\widehat{z} \in \M_{h}$ such that $\ker(E)=  
\menge{\alpha \widehat{z}}{\alpha\in \RR}$ in view of 
\eqref{dim_1_TNI}. Thus, there exists 
$\widehat{\alpha}\in\RR$ such that
$\widehat{m}=\widehat{\alpha}\widehat{z}$ and, since $h^2 \sum_{i,j} 
\widehat{m}_{i,j}=1$, we deduce $\sum_{i,j} 
\widehat{z}_{i,j}\neq 0$. Now, if $\tilde{m}\in  \ker (E)$ 
and $h^2 \sum_{i,j} \tilde{m}_{i,j}=1$, there exists 
$\tilde{\alpha}\in\RR$ such that 
$\tilde{m}=\tilde{\alpha}\widehat{z}$. Hence,
$0=\sum_{i,j} 
(\tilde{m}_{i,j}-\widehat{m}_{i,j})=(\tilde{\alpha}-\widehat{\alpha})
\sum_{i,j} \widehat{z}_{i,j}$, which yields 
$\tilde{\alpha}=\widehat{\alpha}$, implying the uniqueness of 
$\widehat{m}$.

It remains to prove \eqref{e:claim}. Note that it follows from 
\eqref{e:defE} and \eqref{e:adj} that
\begin{equation}
\label{e:Estar}
(\forall (i,j)\in\mathcal{I}_N^2) \quad  (E^*z)_{i,j}=-\nu(\Delta_h z)_{i,j} 
- \langle P_{C}\big(-[D_h\widehat{u} 
]\big)\big)_{i,j}\,|\,[D_h z]_{i,j}\rangle,
\end{equation}
and, therefore, we deduce that $\menge{\alpha 
\mathbf{1}}{\alpha\in \RR}\subset \ker(E^*)$.
Conversely, suppose that there exists a 
nonconstant $z\in \M_{h}$ such that $E^*z=0$, which, from 
\eqref{e:Estar} and \eqref{e:defdiscr}, is equivalent to
\begin{align}
\label{e:30}
(\forall (i,j)\in\mathcal{I}_N^2) \quad  0&=-\nu(\Delta_h z)_{i,j} 
- \langle P_{C}\big(-[D_h\widehat{u} 
]\big)\big)_{i,j}\,|\,[D_h z]_{i,j}\rangle\nonumber\\
&=\frac{\nu}{h}\big((D_1z)_{i-1,j}+(D_2z)_{i,j-1}-(D_1z)_{i,j}
-(D_2z)_{i,j}\big)\nonumber\\
&\quad+(D_1z)_{i,j}[(D_1\widehat{u})_{i,j}]_- +
(D_1z)_{i-1,j}[(D_1\widehat{u})_{i-1,j}]_+\nonumber\\
&\quad+(D_2z)_{i,j}[(D_2\widehat{u})_{i,j}]_- 
+(D_2z)_{i,j-1}[(D_2\widehat{u})_{i,j-1}]_+.
\end{align}
For every $x_{i,j} \in 
\mathbb{T}_{h}^{2}$, set $\mathcal{N}(x_{i,j})=\{x_{i-1,j}, x_{i+1,j}, 
x_{i,j-1},x_{i,j+1}\}$. Since $\mathbb{T}_{h}^{2}$ is a finite set, there 
exist $x_{i_0,j_0}\in \mathbb{T}_{h}^{2}$, $n \in 
\NN\setminus\{0\}$, and 
$\{x_{i_k,j_k}\}_{k=1}^n\subset \mathbb{T}_{h}^{2}$ such that for all 
$k=0, \ldots, n-1$ 
we have $x_{i_{k+1},j_{k+1}} \in \mathcal{N}(x_{i_k,j_k})$, 
$z(x_{i_{k},j_{k}}) <z(x_{i_{k+1},j_{k+1}})$,  and $z(x_{i_{n},j_{n}})\geq 
z(x)$ for all $x\in \mathcal{N}(x_{i_{n},j_{n}})$. Hence,
$(D_{1}z)_{i_n,j_n}\le 0$, $(D_{2}z)_{i_n,j_n}\le 0$, 
$(D_{1}z)_{i_n-1,j_n}\ge 0$, and $(D_{2}z)_{i_n,j_n-1}\ge 0$, and one 
of previous inequalities is strict.
Altogether, since 
\eqref{e:30} in $(i_{n},j_{n})$ can be written as
$$\begin{array}{l}
(D_{1}z)_{i_n,j_n}\big (\nu -h[(D_{1}\widehat{u})_{i_n,j_n}]_{-} 
\big)+(D_{2}z)_{i_n,j_n}\big (\nu -h[(D_{2}\widehat{u})_{i_n,j_n}]_{-} 
\big)\\[8pt]
\hspace{0.3cm}=(D_{1}z)_{i_n-1,j_n}\big (\nu 
+h[(D_{1}\widehat{u})_{i_n-1,j_n}]_{+} \big)+(D_{2}z)_{i_n,j_n-1}\big 
(\nu +h[(D_{2}\widehat{u})_{i_n,j_n-1}]_{+} \big),
\end{array}
$$
we obtain a contradiction, and the proof is complete.
\end{proof}

\begin{remark}
Note that Proposition~\ref{p:uniq} also holds for every
convex differentiable non-local coupling $\Phi_h$.
\end{remark}
\subsubsection{Algorithms}

Now we focus in numerical approaches to solve 
\eqref{discrete_optimization_problem} for several formulations of the 
problem. We start with numerical 
approaches  in \cite{10.condat,partial_inv} and we compare their 
efficiency with the vector subspace technique introduced in this paper.

Note that \eqref{discrete_optimization_problem} is equivalent to
\begin{align}
\label{e:1form}
\min_{(m,w)\in\mathcal{M}_h\times\mathcal{W}_h}F(m,w)+G_1(L_1(m,w))+H(m,w),
\end{align}
where
\begin{equation}
\label{e:deffunctMFG}
\begin{cases}
F\colon (m,w)\mapsto B_h(m,w) +\scal{K_0}{m} \\
G_1=\iota_{\{(0,1)\}}\\
L_1\colon (m,w)\mapsto 
(-\nu\Delta_hm+{{\rm 
div}}_h(w),h^2\scal{\boldsymbol{1}}{m})\\
H\colon (m,w)\mapsto \frac{1}{2}\scal{m}{K_hm}.
\end{cases}
\end{equation}

Observe that 
\begin{equation}
\label{e:nablaH}
\nabla H\colon (m,w)\mapsto 
(K_hm,0)
\end{equation}
is $\|K_h\|-$Lipschitz.
Then the algorithm proposed in \cite{jota1} with $T=P_S$ solves 
\eqref{e:1form}, which
reduces to, for every $k\in\NN$, 
\begin{equation}
\label{mfg_condat_1}
\left\lfloor
\begin{array}{ll}
(p^{k+1},v^{k+1})= 
(p^{k}+\gamma(-\nu{\Delta}_h\overline{m}^{k}
+{\rm div}_h(\overline{w}^{k})),v^{k}+\gamma
h^2\scal{\mathbf{1}}{\overline{m}^{k}}-\gamma)\\
(n^{k+1},z^{k+1})=(m^{k}\!-\tau 
((-\nu{\Delta}_h)^{*}p^{k+1}\!\!+\!h^{2}v^{k+1}\mathbf{1}+K_hm^{k}),w^{k}
\!-\!\tau
({\rm div}_h)^{*}p^{k+1})\\
(m^{k+1},w^{k+1})=\prox_{\tau
F}(n^{k+1},z^{k+1})\\
(\widetilde{m}^{k+1},\widetilde{w}^{k+1})=P_S(m^{k+1},w^{k+1})\\
(\overline{m}^{k+1},\overline{w}^{k+1})=(m^{k+1}+\widetilde{m}^{k+1}-
m^{k},w^{k+1}+\widetilde{w}^{k+1}-w^{k}),
\end{array}
\right.
\end{equation}
where $(m^{0},w^{0})\in\mathcal{M}_h\times\mathcal{W}_h$, 
$(\overline{m}^{0},\overline{w}^{0})=(m^{0},w^{0})$, 
$(p^{0},v^{0})\in\mathcal{M}_h\times \RR$, and $S\supset \arg\min 
(F+G\circ 
L_1+H)$. This method converges 
if $\gamma>0$ and $\tau>0$ satisfy 
\begin{equation}
\|L_{1}\|^{2}<
\dfrac{1}{\gamma}\left(\dfrac{1}{\tau}-\dfrac{\|K_h\|}{2}\right).
\end{equation} 
When $S=\M_h\times\mathcal{W}_h$, \eqref{mfg_condat_1} reduces 
to the method proposed in \cite{10.condat}. As noticed in 
\cite[Remark~4.1]{BAKS}, this algorithm generates unfeasible primal 
sequences leading to slow convergence and it will be not 
considered in our comparisons in Section~\ref{sec:numer}.
To reinforce feasibility, we consider $S=\menge{(m,w)\in\M_h\times 
\mathcal{W}_h}{h^2\scal{\boldsymbol{1}}{m}=1}$ as in \cite{BAKS}. 

An equivalent formulation to \eqref{discrete_optimization_problem} is 
\begin{align}
\label{e:2form}
\min_{(m,w)\in \M_h\times\mathcal{W}_h}F(m,w)+G_2(m,w)+H(m,w),
\end{align}
where $G_2=\iota_{\{(0,1)\}}\circ L_1$.
The formulation in \eqref{e:2form} can also be solved by 
\cite{10.condat} in the case when the linear operator is $\Id$.
Note that $\boldsymbol{1}\in\ker(-\nu{\Delta}_h)$
and $h^2\scal{\boldsymbol{1}}{\boldsymbol{1}}=1$, which yields
$G_2=\iota_{(\ker L_1+(\boldsymbol{1},0))}$
and $\prox_{\gamma 
G_{2}^{*}}\colon (m,w)\mapsto (\Id-P_{\ker 
L_1})(m-\gamma\mathbf{1},w)$
\cite[Theorem~14.3(ii)]{19.Livre1}. Hence, it follows from 
\eqref{e:nablaH} that algorithm in \cite{10.condat} 
reduces to 
\begin{equation}
\label{mfg_condat_2}
(\forall k\in {\mathbb{N}})\quad
\left\lfloor
\begin{array}{ll}
(p^{k+1},\ell^{k+1})= 
(\Id-P_{\ker 
L_1})(p^{k}
+\gamma\overline{m}^{k}-\gamma\mathbf{1},\ell^{k}+\gamma\overline{w}^{k})\\
(m^{k+1},w^{k+1})=\prox_{\tau F}(m^{k}-\tau 
(p^{k+1}+K_hm^{k}),w^{k}-\tau\ell^{k+1})\\
(\overline{m}^{k+1},\overline{w}^{k+1})=(2m^{k+1}-m^{k},2w^{k+1}-w^{k}),
\end{array}
\right.
\end{equation}
where $(m^{0},w^{0})\in\M_h\times\mathcal{W}_h$, 
$(\overline{m}^{0},\overline{w}^{0})=(m^{0},w^{0})$ and 
$(p^{0},\ell^{0})\in\M_h\times\mathcal{W}_h$. In this case, the 
algorithm converges if $\gamma>0$ and $\tau>0$ satisfy
\begin{equation}
\label{e:Hcondgrad}
\gamma<\dfrac{1}{\tau}-\dfrac{\|K_h\|}{2}.
\end{equation}
On the other hand, by defining the 
closed vector subspace
\begin{equation}
V=\ker L_1,
\end{equation}
we have $G_2=\iota_{V+(\boldsymbol{1},0)}$ and, by setting 
$\rho=m-\boldsymbol{1}$,
\eqref{e:2form} is equivalent to
\begin{align}
\label{e:3form}
\min_{(\rho,w)\in V}F(\rho+\mathbf{1},w)+H(\rho+\mathbf{1},w).
\end{align}
This problem can be solved by using the 
algorithm in \cite{partial_inv}. Note that $\nabla 
H(\cdot+(\mathbf{1},0))\colon(\rho,w)\mapsto 
(K_h(\rho+\mathbf{1}),0)$ is $\|K_h\|-$Lipschitz and
the algorithm in \cite[Corollary~5.5]{partial_inv} without
relaxation steps reduces to,
for every $k\in\NN$,
\begin{equation}
\label{mfg_fb_sev}
\left\lfloor
\begin{array}{ll}
(s^{k+1},t^{k+1})=\prox_{\tau F}\big((\rho^{k}+\mathbf{1}+\tau 
z^{k},w^{k}+\tau v^{k})-\tau P_{\ker 
L_{1}}(K_h(\rho^{k}+\mathbf{1}),0)\big)\\
(\rho^{k+1},w^{k+1})=P_{\ker L_{1}}(s^{k+1}-\mathbf{1},t^{k+1})\\
(z^{k+1},v^{k+1})=(z^{k}+(\rho^{k+1}-s^{k+1}+\mathbf{1})
/\tau,v^{k}+(w^{k+1}-t^{k+1})/\tau),
\end{array}
\right.
\end{equation}
where $(\rho^{0},w^{0})\in\ker L_{1}$ and 
$(z^{0},v^{0})\in(\ker L_{1})^{\perp}$. The algorithm converges 
under the condition $0<\tau<2/\|K_h\|$. 

An alternative method for solving \eqref{e:3form} is  our algorithm 
when the linear operator is $\Id$.
For every $\gamma >0$, 
\cite[Proposition~24.8(ii)]{19.Livre1} yields
$
\prox_{ 
H(\cdot+(\mathbf{1},0))/\gamma}\colon 
(\rho,w)\mapsto \big((\Id+ 
K_h/\gamma)^{-1}(\rho+\mathbf{1})-\mathbf{1},w\big)
$
and from
\cite[Theorem~14.3(ii)]{19.Livre1} we obtain
\begin{align}
\prox_{\gamma H(\cdot + (\mathbf{1},0))^{*}}\colon (\rho,w)&\mapsto 
(\rho,w)-\gamma\prox_{
H(\cdot+(\mathbf{1},0))/\gamma}(\rho/\gamma, w/\gamma)\nonumber\\
&=\big(\rho+\gamma \mathbf{1}-\gamma (\Id+ 
K_h/\gamma)^{-1}(\rho/\gamma+\mathbf{1}),0\big).
\end{align}
Hence, by using a similar translation for $\prox_{\tau
F(\cdot+(\mathbf{1},0))}$,
\eqref{alg_opt} in the case when $W$ is the whole space and $T=\Id$ 
reduces to 
\begin{equation}
\label{mfg_cp_sev_1}
(\forall k\in {\mathbb{N}})\quad
\left\lfloor
\begin{array}{ll}
p^{k+1}=p^{k}+\gamma\overline{\rho}^{k}+\gamma 
\mathbf{1}-\gamma\left(\Id+K_h/\gamma\right)^{-1}
\left(p^{k}/\gamma+\overline{\rho}^{k}+\mathbf{1}\right)\\
(s^{k+1},t^{k+1})= \prox_{\tau F}\!\big((\rho^{k}+\mathbf{1}+\tau 
z^{k},w^{k}\!+\tau v^{k})\!-\tau P_{\ker 
L_{1}}(p^{k+1}\!,0)\big)\\
(\rho^{k+1},w^{k+1})=P_{\ker L_{1}}(s^{k+1}-\mathbf{1},t^{k+1})\\
(z^{k+1},v^{k+1})=(z^{k}+(\rho^{k+1}-s^{k+1}+\mathbf{1})/\tau,v^{k}+(w^{k+1}-t^{k+1})/\tau)\\
\overline{\rho}^{k+1}=2\rho^{k+1}-\rho^{k},
\end{array}
\right.
\end{equation}
where $(\rho^{0},w^{0})\in\ker L_{1}$, 
$\overline{\rho}^{0}=\rho^{0}$, 
$(z^{0},v^{0})\in(\ker L_{1})^{\perp}$, and $p^{0}\in\M_h$.
In this context, the algorithm converges for every $\tau>0$ and 
$\gamma>0$ satisfying
$\tau\gamma<1$, in view of \eqref{e:pramcond}. Note that
$P_{\ker L_1}$ can be computed by using 
\cite[Example~29.17(iii)]{19.Livre1}.

For the last formulation of this section, observe that, since $k$ is 
a PDS kernel, it follows from \eqref{e:PDS} that 
the operator $K_h$ defined in \eqref{e:Ks} is positive 
semidefinite and, thus, there 
exists a
symmetric positive semidefinite linear operator 
$K_h^{1/2}\colon \mathcal{M}_h\to\mathcal{M}_h$ such that, for every 
$m\in\mathcal{M}_h$,
$\scal{m}{K_hm}=\sum_{i,j}|(K_h^{1/2}m)_{i,j}|^2$ 
\cite[Theorem~VI.9]{ReedSimon}.
Hence, \eqref{e:3form} can be written 
equivalently as
\begin{align}
\label{e:4form}
\min_{(\rho,w)\in V}F(\rho+\mathbf{1},w)+
G_3(L_2(\rho+\mathbf{1},w)),
\end{align}
where 
\begin{equation}
\label{e:quad}
\begin{cases}
G_3=\frac{1}{2}\|\cdot \|^2,\\
L_2\colon (m,w)\mapsto K_h^{1/2}m.
\end{cases}
\end{equation}
From \cite[Proposition~24.8(i) \& Theorem~14.3(ii)]{19.Livre1}, we 
deduce that, for every $\gamma>0$,
\begin{align*}
\prox_{\gamma G_{3}(\cdot + K_h^{1/2}\mathbf{1})^{*}}\colon \rho
&\mapsto 
\rho-\gamma\prox_{\|\cdot + 
K_h^{1/2}\mathbf{1}\|^2/(2\gamma)}(\rho/\gamma)\nonumber\\
&=(\rho+\gamma 
K_h^{1/2}\mathbf{1})/(1+\gamma).
\end{align*}
Moreover, since $L_{2}^{*}\colon m\mapsto (K_h^{1/2}m,0)$, the 
algorithm in  \eqref{alg_opt} in the case when $W$ is the 
whole space and $T=\Id$  reduces to, for every $k\in\NN$,
\begin{equation}
\label{mfg_cp_sev_2}
\left\lfloor
\begin{array}{ll}
p^{k+1}=(p^{k}+\gamma 
K_h^{1/2}(\overline{\rho}^{k}+\mathbf{1}))/(1+\gamma)\\
(s^{k+1},t^{k+1})= \prox_{\tau F}\big((\rho^{k}+\mathbf{1}+\tau 
z^{k},w^{k}+\tau v^{k})-\tau P_{\ker 
L_{1}}(K_h^{1/2}p^{k+1},0)\big)\\
(\rho^{k+1},w^{k+1})=P_{\ker L_{1}}(s^{k+1}-\mathbf{1},t^{k+1})\\
(z^{k+1},v^{k+1})=(z^{k}+(\rho^{k+1}-s^{k+1}+\mathbf{1})/\tau,v^{k}+(w^{k+1}-t^{k+1})/\tau)\\
\overline{\rho}^{k+1}=2\rho^{k+1}-\rho^{k},
\end{array}
\right.
\end{equation}
where $(\rho^{0},w^{0})\in\ker L_{1}$, 
$\overline{\rho}^{0}=\rho^{0}$, 
$(z^{0},v^{0})\in(\ker L_{1})^{\perp}$, and $p^{0}\in\M_h$. 
In this context, the algorithm converges for every $\tau>0$ and 
$\gamma >0$ satisfying
$\tau\gamma\|K_h^{1/2}\|^2=\tau\gamma\|K_h\|<1$, in view of 
\eqref{e:pramcond}.

\begin{remark}
\label{rem:alg}
Note that, when $\|K_h\|$ is large, algorithm \eqref{mfg_cp_sev_2} 
allows for
a larger set of admissible step-sizes than the algorithm in 
\eqref{mfg_condat_2} in view of
condition \eqref{e:Hcondgrad}. Indeed, \eqref{e:Hcondgrad}
imposes condition $\tau<2/\|K_h\|$ on the primal step-size of 
\eqref{mfg_condat_2}, which affects its efficiency as we will 
see in Section~\ref{sec:numer}.
\end{remark}

In all previous methods we need to compute $\prox_{F}$, where
$F$ is defined in \eqref{e:deffunctMFG}. Note that
\begin{equation}
F\colon (m,w)\mapsto \sum_{i,j=0}^{N-1}f_{i,j}(m_{i,j},w_{i,j}),
\end{equation}
where $f_{i,j}\colon (\eta,\omega)\mapsto 
b(\eta,\omega)+(K_0)_{i,j}\eta$.
We deduce 
from \cite[Corollary~3.1]{BAKS} that $F$
is convex, proper, and lower semicontinuous, for every 
$\gamma>0$,
$\prox_{\gamma F}\colon (m,w)\mapsto (\prox_{\gamma 
f_{i,j}}(m_{i,j},w_{i,j}))_{0\le i,j\le N-1}$, and, for every $0\le i,j\le N-1$,
\begin{equation}
\label{e:proxfinalb}
\prox_{\gamma f_{i,j}}
\colon(\eta,\omega)\mapsto
\begin{cases}
(0,0),\quad&\text{if }\gamma (K_0)_{i,j}\geq \eta+
|P_C\omega|^2/(2\gamma);\\
\Big(p^*,\displaystyle{\frac{p^*}{p^*+\gamma}}P_C\omega\Big),&\text{otherwise},
\end{cases}
\end{equation}
where $p^*> 0$ is the unique solution to 
\begin{equation}
(p+\gamma (K_0)_{i,j}-\eta)
\left(p +\gamma\right)^2-
\frac{\gamma}{2}|P_C\omega|^2=0.
\end{equation}
This computation is also obtained in 
\cite[Proposition~1]{Peyre}
in the case $C=\RR^4$.

\begin{remark}
Observe that algorithms in \eqref{mfg_condat_1}, 
\eqref{mfg_condat_2}, \eqref{mfg_fb_sev}, \eqref{mfg_cp_sev_1},
and \eqref{mfg_cp_sev_2} can also be used in the presence of local 
couplings as those studied in \cite{BAKS}. 
\end{remark}
\subsubsection{Numerical experiments}
\label{sec:numer}
We consider $h\in\{1/20,1/40\}$, the positive definite 
non-local coupling in 
\cite{Achdou_Capuzzo_Dolcetta_2010} given by 
\begin{equation}
K_h=\mu(\Id-\Delta_h)^{-p}
\end{equation}
 for $\mu=10$ and $p=1$, and 
\begin{equation}
(\forall (i,j)\in\mathcal{I}^2_N)\quad 
(K_0)_{i,j}=-\sin(2\pi hj)+\sin(2\pi hi)+\cos(4\pi hi).
\end{equation}
We vary $\nu\in\{0.05,0.2,0.5\}$ and,
for every $h\in\{1/20,1/40\}$, we choose to stop every algorithm 
when the 
$L^2$ norm of the difference between two consecutive iterations 
is less than $5h^3$ or the number of iterations exceeds 3000. 
In Tables~\ref{t:20_05}-\ref{t:40_5}, we report computational 
time, number of iterations,  value of the objective function, and 
residuals of  the constraints at the resulting vector $(m^*,w^*)$.
 
\begin{table}[H]
	\caption{Execution time, number of iterations with error tolerance 
	$5h^{3}$, value of the objective function in solution, and residuals 
	of solution for the case $h=1/20$ and $\nu=0.05$}
	\centering
	\begin{tabular}{c|c|c|c|c|c|}
		\cline{1-6}
		\multicolumn{1}{|c|}{  Alg.}        & time(s) & iter. &  
		obj. value & $\|L_1(m^*,w^*)-(0,1)\|$ & $d^2_C(w^*)$\\ \hline
		\multicolumn{1}{|c|}{\eqref{mfg_condat_1}}   &         0.67  
		& 55
		&   20000  &  3.469  & 0 \\ \hline
		\multicolumn{1}{|c|}{\eqref{mfg_condat_2}}  & 
		32.03
		& 833
		&  20148 & 0.00135
		& 0 \\ \hline
		\multicolumn{1}{|c|}{\eqref{mfg_fb_sev}}  &  
		39.46 & 658
		& 20158 & $5.467\cdot 10^{-13}$
		& $1.518\cdot 10^{-9}$\\ \hline
		\multicolumn{1}{|c|}{\eqref{mfg_cp_sev_1}}  &    43.77
		& 774
		&  20172 & $1.059 \cdot 10^{-12}$
		& $1.379 \cdot 10^{-5}$\\ \hline
		\multicolumn{1}{|c|}{\eqref{mfg_cp_sev_2}} & 12.88 
		& 260
		& 20169  & $3.058 \cdot 10^{-12}$
		& $2.150 \cdot 10^{-8}$\\ \hline
	\end{tabular}
	\label{t:20_05}
\end{table}

\begin{table}[H]
	\caption{Execution time, number of iterations with error tolerance 
	$5h^{3}$, value of the objective function in solution, and residuals 
	of solution for the case $h=1/20$ and $\nu=0.2$}
	\centering
	\begin{tabular}{c|c|c|c|c|c|}
		\cline{1-6}
		\multicolumn{1}{|c|}{  Alg.}        & time(s) & iter. &  
		obj. value & $\|L_1(m^*,w^*)-(0,1)\|$ & $d^2_C(w^*)$\\ \hline
		\multicolumn{1}{|c|}{\eqref{mfg_condat_1}}   &         1.03  
		& 81
		&   20000  &  15.365  & 0 \\ \hline
		\multicolumn{1}{|c|}{\eqref{mfg_condat_2}}  & 
		14.10
		& 312
		&  20076 & 0.00309
		& 0 \\ \hline
		\multicolumn{1}{|c|}{ \eqref{mfg_fb_sev}}  &  
		13.84 & 226
		& 20083 & $5.373\cdot 10^{-13}$
		& $2.688\cdot 10^{-9}$\\ \hline
		\multicolumn{1}{|c|}{\eqref{mfg_cp_sev_1}}  &    47.04
		& 741
		&  20089 & $1.403 \cdot 10^{-12}$
		& $2.779 \cdot 10^{-12}$\\ \hline
		\multicolumn{1}{|c|}{\eqref{mfg_cp_sev_2}} & 5.17 
		& 80
		& 20089 & $5.183 \cdot 10^{-13}$
		& $1.555 \cdot 10^{-8}$\\ \hline
	\end{tabular}
		\label{t:20_2}
\end{table}

\begin{table}[H]
	\caption{Execution time, number of iterations with error tolerance 
	$5h^{3}$, value of the objective function in solution, and residuals 
	of 
	solution for the case $h=1/20$ and $\nu=0.5$}
	\centering
	\begin{tabular}{c|c|c|c|c|c|}
		\cline{1-6}
\multicolumn{1}{|c|}{  Alg.}        & time(s) & iter. &  
		obj. value & $\|L_1(m^*,w^*)-(0,1)\|$ & $d^2_C(w^*)$\\ \hline
		\multicolumn{1}{|c|}{\eqref{mfg_condat_1}}   &         1.01  
		& 87
		&   20000  &  34.964  & 0 \\ \hline
		\multicolumn{1}{|c|}{ \eqref{mfg_condat_2}}  & 
		7.24
		& 161
		&  20014 & 0.00528
		& 0 \\ \hline
		\multicolumn{1}{|c|}{\eqref{mfg_fb_sev}}  &  
		7.18 & 128
		& 20017 & $1.063\cdot 10^{-12}$
		& $2.176\cdot 10^{-10}$\\ \hline
		\multicolumn{1}{|c|}{\eqref{mfg_cp_sev_1}}  &    43.22
		& 741
		&  20021 & $1.127 \cdot 10^{-12}$
		& $1.536 \cdot 10^{-13}$\\ \hline
		\multicolumn{1}{|c|}{ \eqref{mfg_cp_sev_2}} & 4.38 
		& 78
		& 20021 & $1.108 \cdot 10^{-12}$
		& $1.395 \cdot 10^{-11}$\\ \hline
	\end{tabular}
			\label{t:20_5}
\end{table}

\begin{table}[H]
	\caption{Execution time, number of iterations with error tolerance 
	$5h^{3}$, value of the objective function in solution, and residuals 
	of solution for the case $h=1/40$ and $\nu=0.05$}
	\centering
	\begin{tabular}{c|c|c|c|c|c|}
		\cline{1-6}
\multicolumn{1}{|c|}{  Alg.}        & time(s) & iter. &  
		obj. value & $\|L_1(m^*,w^*)-(0,1)\|$ & $d^2_C(w^*)$\\ \hline
		\multicolumn{1}{|c|}{\eqref{mfg_condat_1}}   &         22.32 
		& 462
		&   80000  &  7.365  & 0 \\ \hline
		\multicolumn{1}{|c|}{\eqref{mfg_condat_2}}  & 
		2389.47
		& 3000
		&  80705 & 0.00786
		& 0 \\ \hline
		\multicolumn{1}{|c|}{\eqref{mfg_fb_sev}}  &  
		2776.92 & 1915
		& 80710 & $3.152\cdot 10^{-12}$
		& $5.686\cdot 10^{-11}$\\ \hline
		\multicolumn{1}{|c|}{\eqref{mfg_cp_sev_1}}  &    4189.86
		& 3000
		&  80717 & $4.061 \cdot 10^{-12}$
		& $3.258 \cdot 10^{-6}$\\ \hline
		\multicolumn{1}{|c|}{ \eqref{mfg_cp_sev_2}} & 751.18 
		& 695
		& 80717  & $3.757 \cdot 10^{-12}$
		& $9.748 \cdot 10^{-9}$\\ \hline	
	\end{tabular}
			\label{t:40_05}
\end{table}

\begin{table}[H]
	\caption{Execution time, number of iterations with error tolerance 
	$5h^{3}$, value of the objective function in solution, and residuals 
	of solution for the case $h=1/40$ and $\nu=0.2$}
	\centering
	\begin{tabular}{c|c|c|c|c|c|}
		\cline{1-6}
\multicolumn{1}{|c|}{  Alg.}        & time(s) & iter. &  
		obj. value & $\|L_1(m^*,w^*)-(0,1)\|$ & $d^2_C(w^*)$\\ \hline
		\multicolumn{1}{|c|}{\eqref{mfg_condat_1}}   &         38.55 
		& 727
		&   80000  &  27.414  & 0 \\ \hline
		\multicolumn{1}{|c|}{\eqref{mfg_condat_2}}  & 
		533.00
		& 727
		&  80369 & 0.00770
		& 0 \\ \hline
		\multicolumn{1}{|c|}{ \eqref{mfg_fb_sev}}  &  
		636.59 & 461
		& 80372 & $4.219\cdot 10^{-12}$
		& $1.298\cdot 10^{-10}$\\ \hline
		\multicolumn{1}{|c|}{\eqref{mfg_cp_sev_1}}  &    1387.21
		& 950
		&  80376 & $4.708 \cdot 10^{-12}$
		& $5.260 \cdot 10^{-12}$\\ \hline
		\multicolumn{1}{|c|}{\eqref{mfg_cp_sev_2}} & 136.66 
		& 119
		& 80375  & $4.491 \cdot 10^{-12}$
		& $2.492 \cdot 10^{-9}$\\ \hline
	\end{tabular}
				\label{t:40_2}
\end{table}

\begin{table}[H]
	\caption{Execution time, number of iterations with error tolerance 
	$5h^{3}$, value of the objective function in solution, and residuals 
	of solution for the case $h=1/40$ and $\nu=0.5$}
	\centering
	\begin{tabular}{c|c|c|c|c|c|}
		\cline{1-6}
\multicolumn{1}{|c|}{  Alg.}        & time(s) & iter. &  
		obj. value & $\|L_1(m^*,w^*)-(0,1)\|$ & $d^2_C(w^*)$\\ \hline
		\multicolumn{1}{|c|}{\eqref{mfg_condat_1}}   &         
		37.06 
		& 724
		&   80000  &  70.955  & 0 \\ \hline
		\multicolumn{1}{|c|}{\eqref{mfg_condat_2}}  & 
		286.23
		& 387
		&  80082 & 0.0296
		& 0 \\ \hline
		\multicolumn{1}{|c|}{\eqref{mfg_fb_sev}}  &  
		428.79 & 251
		& 80083 & $8.392\cdot 10^{-12}$
		& $2.099\cdot 10^{-11}$\\ \hline
		\multicolumn{1}{|c|}{\eqref{mfg_cp_sev_1}}  &    2036.42
		& 950
		&  80085 & $8.279 \cdot 10^{-12}$
		& $1.335 \cdot 10^{-12}$\\ \hline
		\multicolumn{1}{|c|}{\eqref{mfg_cp_sev_2}} & 111.39 
		& 100
		& 80085  & $8.408 \cdot 10^{-12}$
		& $2.984 \cdot 10^{-12}$\\ \hline
	\end{tabular}
	\label{t:40_5}
\end{table}

Observe that, in all cases, algorithms \eqref{mfg_condat_1} and 
\eqref{mfg_condat_2} stop at iterates which are far from the 
solution, since the 
residual $\|L_1(m^*,w^*)-(0,1)\|$ is far away from $0$ for the 
chosen 
tolerance. This residual is larger as viscosity increases. 
In contrast, for the same tolerance, the vector subspace based 
algorithms \eqref{mfg_fb_sev},
\eqref{mfg_cp_sev_1}, and \eqref{mfg_cp_sev_2} achieve
iterates with negligible residuals. Among the latter, 
our proposed algorithm \eqref{mfg_cp_sev_2} is the most efficient 
in terms of computational time and number of iterations.
We explain this good behavior  by the fact that 
\eqref{mfg_cp_sev_2} takes full 
advantage of the convex quadratic cost by splitting $K_h^{1/2}$ 
from 
$\|\cdot\|^2/2$ in its architecture (see 
\eqref{e:quad}).
A reason for this improvement is the larger step-sizes that 
this algorithm can take, as stated in Remark~\ref{rem:alg}.
In the case of 
more general couplings, algorithms 
\eqref{mfg_fb_sev} and \eqref{mfg_cp_sev_1} are also efficient 
alternatives to solve \eqref{discrete_optimization_problem}. 

In the Figures~\ref{fig:002}-\ref{fig:05}, we illustrate the solution 
obtained from algorithm 
\eqref{mfg_cp_sev_2} for different values of $\nu$ and $h$.
\begin{figure}[H]
	\centering
\begin{subfigure}{.5\textwidth}
  \centering
  \includegraphics[width=0.9\linewidth]{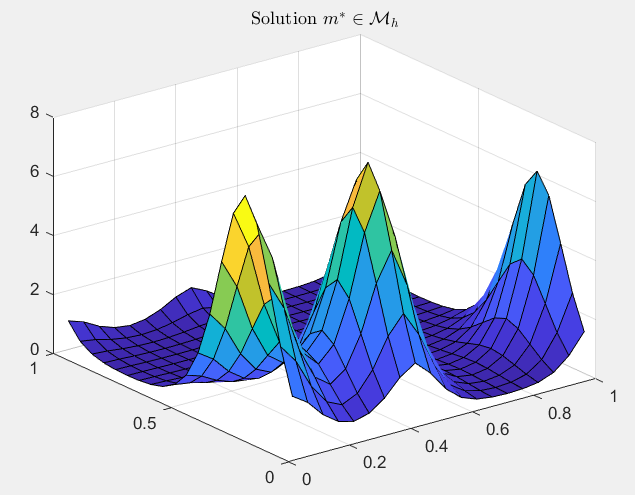}
\end{subfigure}%
\begin{subfigure}{.5\textwidth}
  \centering
  \includegraphics[width=0.9\linewidth]{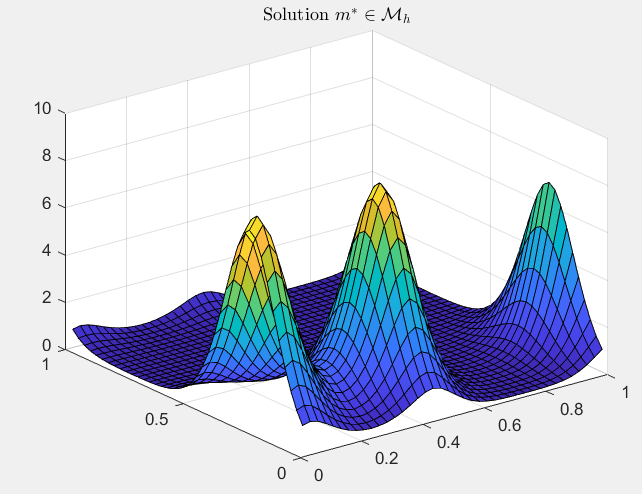}
\end{subfigure}%
	\caption{Solution $m^{*}\in\mathcal{M}_{h}$ to 
	\eqref{discrete_optimization_problem} for $\nu=0.05$ with 
	$h=1/20$ 
		(left) and $h=1/40$ (right)}
\label{fig:002}
\end{figure}

\begin{figure}[H]
	\centering
\begin{subfigure}{.5\textwidth}
  \centering
  \includegraphics[width=0.9\linewidth]{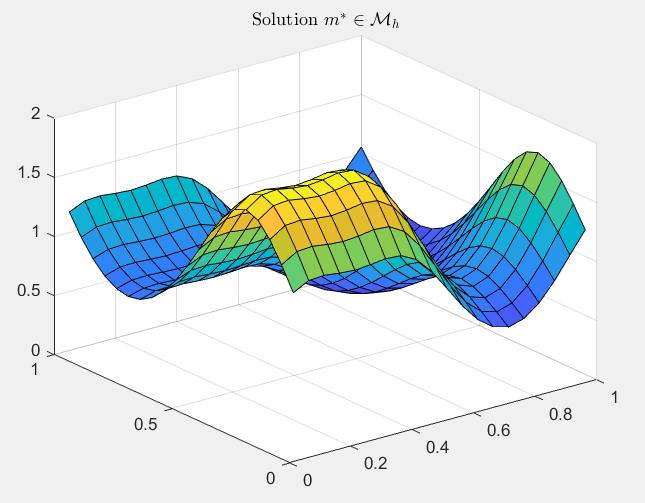}
\end{subfigure}%
\begin{subfigure}{.5\textwidth}
  \centering
  \includegraphics[width=0.9\linewidth]{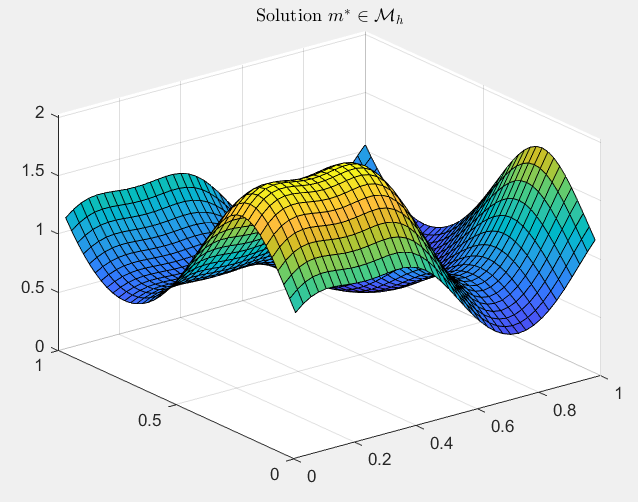}
\end{subfigure}%
	\caption{Solution $m^{*}\in\mathcal{M}_{h}$ to 
	\eqref{discrete_optimization_problem} for $\nu=0.2$ with 
$h=1/20$ 
		(left) and $h=1/40$ (right)}
\label{fig:02}
\end{figure}

\begin{figure}[H]
	\centering
\begin{subfigure}{.5\textwidth}
  \centering
  \includegraphics[width=0.9\linewidth]{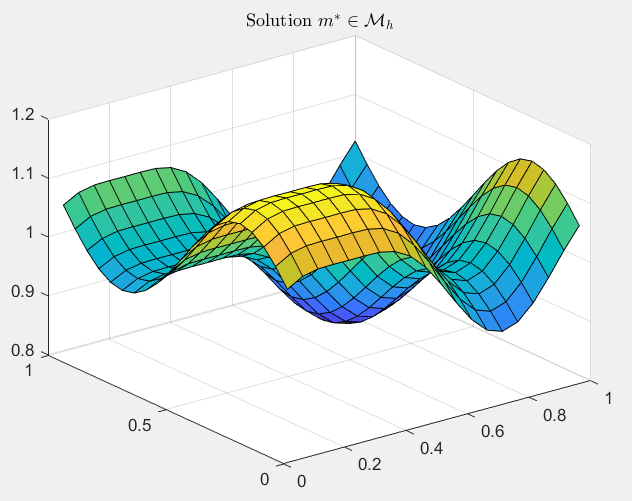}
\end{subfigure}%
\begin{subfigure}{.5\textwidth}
  \centering
  \includegraphics[width=0.9\linewidth]{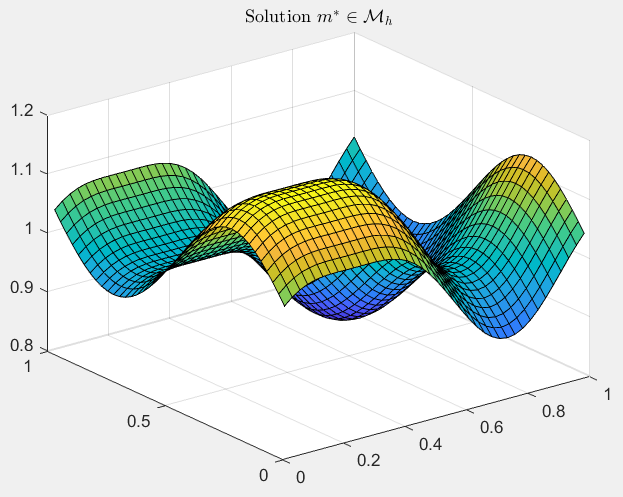}
\end{subfigure}%
	\caption{Solution $m^{*}\in\mathcal{M}_{h}$ to 
	\eqref{discrete_optimization_problem} for $\nu=0.5$ with 
$h=1/20$ 
		(left) and $h=1/40$ (right)}
\label{fig:05}
\end{figure}

\section{Conclusions}
\label{conclusion_paper2}
We propose a primal-dual method with partial inverse for solving 
constrained composite monotone inclusions involving a normal cone 
to a closed vector subspace. When the monotone operators are 
subdifferentials of convex functions, our method solves composite 
convex optimization problems over closed vector subspaces.
We also incorporate a priori information on the solution of the 
monotone inclusion, which produces an additional projection step in 
the primal-dual algorithm. Either this projection or our vector 
subspace approach produces significant gains in numerical 
efficiency with 
respect to the available methods in the literature.

\textbf{Acknowledgements}
The work of the first and second authors are founded by the 
National Agency for Research and Development (ANID) under grants 
FONDECYT 1190871 and FONDECYT 11190549, respectively. The 
third author is founded by the
Scholarship program CONICYT-PFCHA/MagísterNacional/2019 - 
22190564 and FONDECYT 1190871 of ANID. 

\printbibliography%

\end{document}